\documentclass[12pt]{amsart}

\usepackage{microtype}

\usepackage{amsmath}
\usepackage{amssymb}
\usepackage{amsfonts}
\usepackage{amsthm}
\usepackage{amscd}

\usepackage{enumerate}
\usepackage{multicol}
\usepackage{hyperref}
\usepackage{url}
\usepackage{array}
\usepackage{xcolor}

\colorlet{darkgreen}{green!70!black}
\colorlet{darkred}{red!85!black}

\renewcommand{\i}{\textcolor{darkgreen}{1}}
\renewcommand{\o}{\textcolor{blue}{0}}
\newcommand{\q}{\textcolor{darkred}{?}}
\newcommand{\ii}{{1}}
\newcommand{\oo}{{0}}
\newcommand{\qq}{{?}}

\newcommand{\st}{{*}}
\newcommand{\tw}{\widehat{\st\st}}

\usepackage{graphicx}
\usepackage{subfig}

\usepackage{tikz,pgfplots}
\usetikzlibrary{matrix}
\usetikzlibrary{backgrounds}
\usetikzlibrary{shapes.multipart}
\usetikzlibrary{arrows,positioning}

\graphicspath{{Figures/}}

\def \Z {\mathbb{Z}}
\def \N {\mathbb{N}}

\def \P {\mathbb{P}}

\def \w {w}

\newcommand{\ques}{{?}}

\newcommand{\tsigma}{\tilde{\sigma}}

\newcommand{\altleq}{\trianglelefteq}
\newcommand{\altless}{\vartriangleleft}
\newcommand{\altgreater}{\vartriangleright}
\newcommand{\altpreceq}{\trianglelefteq}

\renewcommand{\tilde}{\widetilde}

\newcommand{\Out}{\operatorname{Out}}
\newcommand{\In}{\operatorname{In}}

\newcommand{\newmod}{\operatorname{mod}}

\newcommand{\ch}[2]{\ensuremath{\left( \begin{smallmatrix} #1 \\ #2
\end{smallmatrix} \right)}}

\newcommand{\df}{\textbf}

\newtheorem{lem}{Lemma}
\newtheorem{prop}{Proposition}
\newtheorem{thm}{Theorem}

\theoremstyle{definition}
\newtheorem{exa}{Example}

\numberwithin{equation}{section}
\numberwithin{lem}{section}
\numberwithin{prop}{section}
\numberwithin{exa}{section}
\numberwithin{rem}{section}

\title[Percolation Games]{Percolation games, probabilistic cellular automata, and the hard-core model}
\author[Holroyd]{Alexander E.\ Holroyd}
\address{Alexander E.\ Holroyd}
\email{holroyd@uw.edu}
\author[Marcovici]{Ir\`ene Marcovici}
\address{Ir{\`e}ne Marcovici,
Institut Elie Cartan de Lorraine\\
Universit\'e de Lorraine,
Campus Scientifique, BP 239,
54506 Vandoeuvre-l\`es-Nancy Cedex, France}
\email{irene.marcovici@univ-lorraine.fr}
\author[Martin]{James B.\ Martin}
\address{James B.\ Martin,
Department of Statistics,
University of Oxford,
1~South Parks Road,
Oxford OX1 3TG, UK}
\email{martin@stats.ox.ac.uk}

\date{16 February 2018}


\calclayout

\begin{document}
\begin{abstract}
Let each site of the square lattice $\Z^2$ be independently assigned
one of three states:
a \textit{trap} with probability $p$, a \textit{target} with
probability $q$, and \textit{open} with probability $1-p-q$,
where $0<p+q<1$.
Consider the following game:
a token starts at the origin, and two players take turns to move, where a move consists of moving the token
from its current site $x$ to either $x+(0,1)$ or $x+(1,0)$.
A player who moves the token to a trap loses the game immediately,
while a player who moves the token to a target wins the game immediately. Is there positive
probability that the game is \emph{drawn} with best play -- i.e.\ that
neither player can force a win?  This is equivalent to the question of
ergodicity of a certain family of elementary one-dimensional probabilistic cellular automata (PCA).  These automata
have been studied in the contexts of enumeration of
directed lattice animals, the golden-mean subshift, and the hard-core model, and their ergodicity has been noted as an open problem by several authors.  We prove that these PCA are ergodic, and correspondingly that
the game on $\Z^2$ has no draws.

On the other hand, we prove that certain analogous games \emph{do} exhibit draws for suitable parameter values on various directed graphs in higher dimensions, including an oriented version of the even sublattice of $\Z^d$ in all $d\geq3$. This is proved via a dimension reduction to a hard-core lattice gas in
dimension $d-1$.  We show that draws occur whenever the corresponding
hard-core model has multiple Gibbs distributions.  We conjecture that draws
occur also on the standard oriented lattice $\Z^d$ for $d\geq 3$, but here
our method encounters a fundamental obstacle.
\end{abstract}
\maketitle

\section{Introduction}
We introduce and study \textbf{percolation games}
on various graphs. For the lattice $\Z^2$, we show that the probability of a draw is 0; this is equivalent to proving ergodicity for a certain family of
probabilistic cellular automata. In higher dimensions, we prove
that draws can occur, by developing a connection
to the question of multiplicity of Gibbs distributions for
the hard-core model.


\subsection{Two dimensional games and ergodicity} \label{sec:intro2d}
Let $p,q\in[0,1)$ with $0\leq p+q\leq 1$.
Let each site of $\Z^2$ be one of three types: a \textbf{trap} with probability $p$,
a \textbf{target} with probability $q$, and
\textbf{open} with probability $1-p-q$, independently for different
sites. Consider the following two-player game. A token
starts at the origin. The players move alternately; if the token is currently at $x$,
a move consists of moving it to
$x+(0,1)$ or to $x+(1,0)$. If a player moves the token to a
trap, that player loses the game immediately.
If a player moves the token to a target, that player wins the game immediately.
Otherwise (i.e.\ if the destination site is open), the game continues with the other player's turn.

The entire random assignment of traps, targets and open sites to $\Z^2$ (which we call the \df{percolation configuration}) is known to both players at all times.  We call this game the \df{percolation game} on $\Z^2$.  We will call the special case $q=0$ (where we have only traps and open sites) the \textbf{trapping game}, and the case $p=0$ (where we have only targets and open sites) the \textbf{target game}.

\begin{figure}[!]
\centering
{\quad
\includegraphics[width=.4\textwidth]{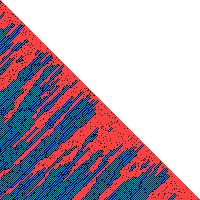}
\hfill
\includegraphics[width=.4\textwidth]{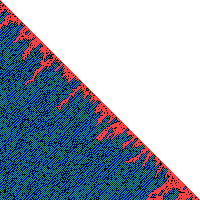}
\quad}
\caption{Outcomes of the trapping game ($q=0$) on the finite region
$\{x\in\Z_+^2: x_1+x_2\leq 200\}$,
 declaring a draw if the token reaches the diagonal $x_1+x_2=200$.
 On the left $p=0.1$ and on the right $p=0.2$.
Traps are coloured black, and otherwise
colours indicate the outcome when the game is
started from that site: first player win (blue); first player loss (green), or draw (red).}
\label{sim}
\end{figure}

If $p+q\geq 1-p_c$, where $p_c$ is the critical probability
for directed site percolation, then, with probability 1,
only finitely many sites can be reached from the origin
along directed paths of open sites, and so the game must
end in finite time.  In particular, one or other player
must have a winning strategy. (A \df{strategy} for one or
other player is a map that assigns a legal move, where one
exists, to each vertex; a \df{winning} strategy is one that
results in a win for that player, whatever strategy the
other player uses.) Suppose on the other hand that
$p+q<1-p_c$; is there now a positive probability that neither
player has a winning strategy? In that case we say that the
game is a \df{draw}, with the interpretation that it
continues for ever with best play. (When $p=q=0$ the game is
clearly always a draw.)

See Figure~\ref{sim} for simulations on a finite triangular
region, with draws imposed as a boundary condition.  As the
size of this region tends to $\infty$, the probability of a
draw starting from the origin converges to the probability
of a draw on $\Z^2$; the question is whether this limiting
probability is positive for any $p$ and $q$.

Related questions are considered in \cite{holroyd-martin}
and \cite{BHMW}, in which the underlying graph is
respectively a Galton-Watson tree, and a random subset of
the lattice with undirected moves.

In our case of a random subset of $\Z^2$ with directed
moves, the outcome (first-player win, first-player loss, draw) of the game started from each site can be interpreted in terms of the evolution
of a certain one-dimensional discrete-time probabilistic
cellular automaton (PCA); the state of the PCA at a given
time relates to the outcomes associated to the sites on a
given Northwest-Southeast diagonal of $\Z^2$.

The PCA has alphabet $\{0,1\}$ and universe $\Z$, so that a
configuration at a given time is an element of
$\{0,1\}^\Z$.  (The three game outcomes will correspond to
the two states of the PCA via a coupling of two copies of
the PCA.) The evolution of the PCA is as follows. Given a
configuration $\eta_t$ at some time $t$, the configuration
$\eta_{t+1}$ at time $t+1$ is obtained by updating each
site $n\in\Z$ simultaneously and independently, according
to the following rule.
\begin{samepage}
\begin{itemize}
\item
If $\eta_t(n-1)=\eta_t(n)=0$, then
$\eta_{t+1}(n)$ is set to $0$ with probability $p$
and $1$ with probability $1-p$.
\item
Otherwise (i.e.\ if at least one of $\eta_t(n-1)$ and $\eta_t(n)$
is $1$), $\eta_{t+1}(n)$ is set to $0$ with probability $1-q$ and $1$ with
probability $q$.
\end{itemize}
\end{samepage}
We denote this PCA $A_{p,q}$.
Its evolution rule at each site is illustrated in Figure
\ref{figure-A}.  (The time coordinate $t$ increases from
top to bottom, and the spatial coordinate $n$ increases
from left to right).
%
%
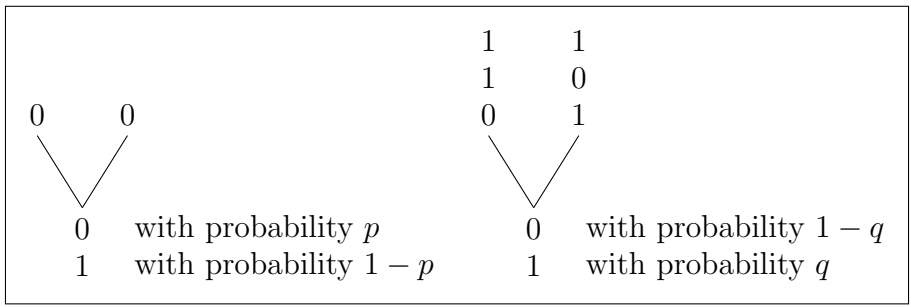
\begin{figure}
\centering
\begin{tikzpicture}[scale=.6,every text node part/.style={align=left},framed]
\begin{scope}
\node[above] (l) at (-1,1) {0};
\node[above] (r) at (1,1) {0};
\node[below] (b) at (0,-.6) {0\\1};
\node[right =3mm of b] (bb) {with~probability~$p$\\with~probability~$1-p$};
\draw (l.south)--(b.north);
\draw (r.south)--(b.north);
\end{scope}
\begin{scope}[xshift=10cm]
\node[above] (l) at (-1,1) {1\\1\\0};
\node[above] (r) at (1,1) {1\\0\\1};
\node[below] (b) at (0,-.6) {0 \\1};
\node[right =3mm of b] (bb) {with~probability~$1-q$\\with~probability~$q$};
\draw (l.south)--(b.north);
\draw (r.south)--(b.north);
\end{scope}
\end{tikzpicture}
\caption{The probabilistic cellular automaton (PCA) $A_{p,q}$.}\label{figure-A}
\end{figure}

Formally, we take $A_{p,q}$ to be the operator on the set of
distributions on $\{0,1\}^\Z$ representing the action of
the PCA; if $\mu$ is the distribution of a configuration in
$\{0,1\}^\Z$, then $A_{p,q}\mu$ is the distribution of the
configuration obtained by performing one update step of the
PCA. A \df{stationary distribution} (or \df{invariant distribution}) of a PCA $F$ is a
distribution $\mu$ such that $F\mu=\mu$. (More generally,
$\mu$ is \df{$k$-periodic} if $F^k\mu=\mu$, and
\df{periodic} if it is $k$-periodic for some $k\geq 1$.) A
PCA is said to be \df{ergodic} if it has a unique
stationary distribution and if from any initial
distribution, the iterates of the PCA converge to that
stationary distribution (in the sense of convergence in
distribution with respect to the product topology).

The PCA $A_{p,q}$ has already been studied from a number of different perspectives.
It is closely related to the enumeration of directed
lattice animals, which are classical objects in combinatorics.
More precisely, if $\mu$ is an invariant distribution of $A_{p,q}$, then its marginals satisfy
the same recursions as the counting series of directed animals on the square lattice, enumerated according to their area and perimeter.
The link was originally made by Dhar \cite{dhar}, and subsequent work
includes \cite{bousquet, leborgne} -- see also Section 4.2 of the survey of
Mairesse and Marcovici \cite{mairesse_tcs} for a short introduction.

It is quite easy to show that the percolation game has positive probability of a draw if and only if $A_{p,q}$ is non-ergodic (see Proposition \ref{prop-draws} below).
It is also easy to see that $A_{p,q}$ is ergodic
whenever $p+q$ is sufficiently large.
The question of whether $A_{p,q}$
is ergodic for \emph{all} $p$ and $q$ has been mentioned as an open problem by several authors -- see in particular
\cite{dobrushin}, as well as discussions in
\cite{leborgne} and \cite{mairesse_tcs}.

%
%

PCA that are defined on $\Z$ and whose alphabet and
neighbourhood are both of size $2$ are sometimes called
\emph{elementary} PCA. A variety of tools have been
developed to study their ergodicity. Under the additional
assumption of left-right symmetry of the update rule, these
PCA are defined by only three parameters: the probabilities
to update a cell to state $1$ if its neighbourhood is in
state $00$, $11$, or $01$ (which is the same as for $10$).
Existing methods can be used to handle more than $90\%$ of
the volume of the cube $[0,1]^3$ defined by this parameter
space, but when $p$ and $q$ are small, the PCA $A_{p,q}$
belongs to an open domain of the cube where none of the
previously known criteria hold~\cite[Chapter 7]{dobrushin}.

We now state our first main result.
\begin{thm}
\label{thm:2d} If $p>0$ or $q>0$ then the PCA $A_{p,q}$ is ergodic,
and the probability of a draw is zero for the percolation game on $\Z^2$.
\end{thm}

We prove ergodicity by considering the \df{envelope} PCA corresponding to
$A_{p,q}$, which is a PCA with an expanded alphabet $\{0,\ques, 1\}$.
The envelope PCA corresponds to the status of the game started from each
site (with the symbols $0$, $\ques$ and $1$ corresponding to wins, draws and
losses respectively).  An evolution of the envelope PCA can be used to encode
a coupling of two copies of the original PCA, with a $\ques$ symbol denoting
sites where the two copies disagree.  We introduce a new method involving a
positive weight assigned to each $\ques$ symbol (whose value depends on the
states of nearby sites). The correct choice of weights is delicate and
non-obvious. We show that if the process is translation-invariant, then the
average weight per site strictly decreases under the evolution of the
envelope PCA, unless it is 0. It follows that any translation-invariant
stationary distribution for the envelope PCA has no $\ques$ symbols, with
probability 1, and from this we will be able to deduce that the game has no
draws with probability 1, so that the original PCA is ergodic. Although the
proof of ergodicity could be phrased so as not to refer to games, the notion
is useful as a semantic tool and a guide to intuition.



In the particular case $q=0$ (corresponding to the trapping game),
it was already known that the PCA has an invariant distribution
which is Markovian in space
which has made it
possible to compute the generating function of directed animals enumerated according to their area alone).
This case also has strong connections to the hard-core
lattice gas model in statistical physics (which has various applications, for example to the modeling of communications networks)
-- see Section \ref{sec:hardcore} of this paper. The
case $(p,q)=(1/2,0)$ in particular relates to the measure of maximal
entropy of the golden mean subshift in dynamical systems -- see
\cite{eloranta} and also~\cite{mairesse_marcovici_IJFCS17}.
A link between the hard-core PCA $A_{p,0}$ and
the trapping game was already mentioned by \cite{leborgne}.
As far as we know, the ergodicity of $A_{p,0}$ has not
been previously observed. It is a particular case of our
Theorem \ref{thm:2d}, but it also follows from
simpler methods which are a special case of those discussed
in Section \ref{sec:converse}.

Indeed, for the trapping game, combining the ergodicity
of $A_{p,0}$ with the Markovian description of the
invariant distribution permits
an explicit description of the distribution of game outcomes along a
diagonal, as a Markov chain.
Consequently, we show that the probability that
the first player wins the trapping game is
\begin{equation}\label{win-p}
\frac{1-2p+\sqrt{\frac{p}{4-3p}}}{2(1-p)}.
\end{equation}
See Figure \ref{curve} for a plot of this winning probability against $p$.
The probability is greater than $1/2$ if and only if $p\in(0,1/3)$,
and its maximum value is $4-2\surd 3=0.5358...$, attained at
$p=(2-\surd 3)/3=0.0893...$.

These methods seemingly do not extend to the case of positive $q$,
and we do not know
an explicit expression for the win probability
(even for the case of the target game, where $p=0$ and $q>0$).
Extending further, for a mis\`ere version of the
trapping game
(see Question \ref{subsection:misere} in the final section
of the paper) there is apparently no similar connection to a PCA
with alphabet $\{0,1\}$, and we do not
have a proof that the probability of a draw is 0.
This is somewhat reminiscent
of the situation for sums of combinatorial games~\cite{onag}, where the
well-developed theory of ``normal play'' games extends only in very limited
cases to their mis\`ere cousins.
%

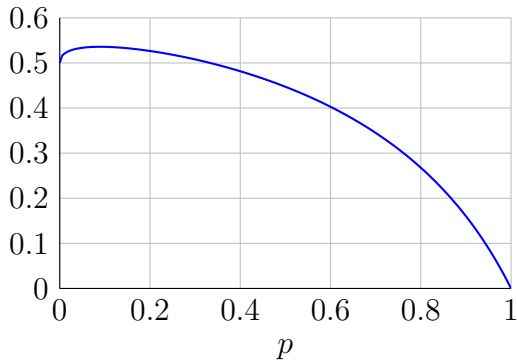
\begin{figure}[htb]
\centering
\begin{tikzpicture}
\begin{axis}[grid=major,major tick style = transparent,axis lines*=left, xlabel={$p$},
axis line style={-},xmin=0,xmax=1,ymin=0,ymax=.6,ytick={0,0.1,...,0.7},x=6cm,y=6cm]
\pgfplotsset{samples=200}
 \addplot[thick, domain=0:0.999,smooth,blue]
(\x,{(1-2*\x+sqrt(\x/(4-3*\x)))/(2-2*\x)});
\end{axis}
\end{tikzpicture}
\caption{The probability that the first player wins the trapping game with $q=0$,
conditional on the origin being open, as a function of the density $p$ of traps.
}
\label{curve}
\end{figure}

\subsection{The trapping game in higher dimensions and the hard-core model}
\label{sec:introhardcore}
We now consider the particular case $q=0$, and
explore extensions of
the trapping game, described above for $\Z^2$, to more
general directed graphs and in particular to lattices in
higher dimensions. Theorem \ref{thm:2d} tells us that in
two dimensions, the probability of a draw is $0$ for all
positive $p$, but we find a very different picture in three
and higher dimensions.

Let $G=(V,E)$ be a locally finite directed graph.
For $x\in V$, let $\Out(x)$ and $\In(x)$ be the
sets of out-neighbours and in-neighbours of $x$ respectively.
For the trapping game on $G$,
let each vertex $x$ be
a trap with probability $p$ and open
with probability $1-p$,
independently for different vertices.
A token starts at some vertex, and the two players
move alternatively; if the token is currently at $x$, a
move consists of moving it to any vertex in $\Out(x)$.
The token is only allowed to move to open sites; if
all the vertices in $\Out(x)$ are traps, then the player
to move loses the game.

For graphs $G$ with an appropriate structure, we develop a
connection to the hard-core model on a related undirected
graph in one fewer dimensions, to obtain a criterion under
which the game is drawn with positive probability.

For an undirected graph with vertex set $W$, and any $\lambda>0$, a \df{Gibbs
distribution} for the \df{hard-core model} on $W$ with \df{activity}
$\lambda$ is a probability distribution on configurations $\eta\in\{0,1\}^W$
such that
\begin{equation}\label{hardcoredef}
\P\Bigl(\eta(v)=1 \, \Bigm| \, (\eta(w): w\ne v)\Bigr) =
\begin{cases}
\displaystyle\frac{\lambda}{1+\lambda},
&\parbox{8em}{if $\eta(w)=0$ for all\\ neighbours $w$ of $v$;}\\[1em]
0,
&\text{otherwise.}
\end{cases}
\end{equation}
Any such Gibbs distribution is concentrated on configurations $\eta$ that
correspond to independent sets, in the sense that no two neighbouring
vertices $v$ and $w$ have $\eta(v)=\eta(w)=1$. If $W$ is finite, then there
is a unique Gibbs distribution, which is the probability distribution that
puts weight proportional to $\prod_{v\in W}\lambda^{\eta(v)}$ on each
configuration $\eta$ that is supported on an independent set. However, for
infinite graphs, there may be multiple Gibbs distributions. A well-known
example is the lattice $\Z^d$ with nearest-neighbour edges. For $d=1$, there
is a unique Gibbs distribution for all activities $\lambda$. However, for
$d\geq 2$, there exist multiple Gibbs distributions when $\lambda$ is
sufficiently large \cite{dobrushin65}.

Returning to the trapping game on a directed graph $G$, we now give the
key assumptions on $G$ that are required for our dimension reduction method.
Suppose there is a partition $(S_k: k\in \Z)$ of the vertex set $V$ of $G$,
and an integer $m\geq 2$, such that the following conditions hold.
\begin{samepage}
\begin{itemize}
\item[(A1)] For all $x\in S_k$, we have $\Out(x)\subset
    S_{k+1}\cup\cdots \cup S_{k+m-1}$.
\item[(A2)] There is a graph automorphism $\phi$ of $G$
    that maps $S_k$ to $S_{k+m}$ for every $k$, and such
    that $\Out(x)=\In(\phi(x))$ for all $x$.
\end{itemize}
\end{samepage}

Then let $D_k$ be the graph with vertex set
$S_k\cup\dots\cup S_{k+m-1}$, with an undirected edge
$(x,y)$ whenever $(x,y)$ is a (directed) edge of $V$.
(Below for convenience we will also use $D_k$ to denote the
vertex set $S_k\cup\dots\cup S_{k+m-1}$.) It is
straightforward to show that under conditions (A1) and
(A2), the graphs $D_k$ are isomorphic to each other for all
$k\in\Z$ (see Lemma \ref{lemma:Dk}); write $D$ for a
generic graph which is isomorphic to any of the $D_k$.
Observe that $D$ is an $m$-partite graph.
We have the following criterion for positive probability of
draws.

\begin{thm}
\label{thm:hardcoredraw} Suppose that the directed graph $G$ satisfies (A1)
and (A2).
If there exist multiple Gibbs distributions for the hard-core model
on $D$ with activity $\lambda$, then the trapping game on $G$ with
$p=1/(1+\lambda)$ has positive probability of a draw from some vertex.
\end{thm}

The simplest case in which to understand the conditions (A1) and (A2) is when
$G$ is the directed lattice $\Z^2$, with $\Out(x)=\{x+(1,0), x+(0,1)\}$ (the
setting of Theorem \ref{thm:2d} in Section \ref{sec:intro2d}). Then we may
take the partition of $\Z^2$ into Northeast-Southwest diagonals given by
$S_k:= \{(x_1,x_2): x_1+x_2=k\}$, along with the bijection $\phi(x)=x+(1,1)$,
and $m=2$. The graph $D$ then consists of the vertices of two successive
diagonals, and is thus isomorphic to the line $\Z$. (In the context of PCA,
$D$ is sometimes called the \emph{doubling graph}.)

We don't know whether the converse statement to Theorem \ref{thm:hardcoredraw}
holds in general -- i.e.\ whether uniqueness of the Gibbs measure on $D$
implies probability $0$ of a draw. In the case $m=2$, this converse
statement does indeed hold -- see
Section \ref{sec:converse} for discussion.

As noted above, there is a unique Gibbs distribution for the hard-core
model on $\Z$ for all $\lambda>0$. In this case $m=2$, and so the converse
statement to  Theorem~\ref{thm:hardcoredraw} says that
there are no draws
for any $p\in(0,1)$. This gives an alternative (and perhaps simpler) proof of Theorem \ref{thm:2d} in the
special case $q=0$.

In higher dimensions the picture is different. We will give several examples
of relevant graphs in Section \ref{sec:hardcore-general} and Theorem
\ref{thm:higherd} below. For the current discussion, consider the case where
$G$ has vertex set $\Z^d_{\text{even}}:=\{x\in\Z^d: \sum x_i \text{ is
even}\}$, with directed edges given by $\Out(x):=\{x\pm e_i+e_d: 1\leq i\leq
d-1\}$ (where $e_i$ is the $i$th standard basis vector in $\Z^d$). So
$\Out(x)$ has size $2(d-1)$; any move of the game increases the $d$th
coordinate by $1$ and also changes exactly one of the other coordinates by
$1$ in either direction. In two dimensions, this game is isomorphic to the
original game on $\Z^2$. For general $d$, conditions (A1) and (A2) hold with
$m=2$ if we set $S_k=\{x\in\Z^d_\text{even}:x_d=k\}$ and $\phi(x)=x+2e_d$.
One then finds that $D$ is isomorphic to the standard $(d-1)$-dimensional
cubic lattice $\Z^{d-1}$ with nearest-neighbour edges. As mentioned above,
there are multiple Gibbs distributions for the hard-core model on $\Z^{d-1}$
whenever $d\geq 3$ and $\lambda$ is large enough; then Theorem
\ref{thm:hardcoredraw} tells us that the trapping game on $G$ has positive
probability of a draw when $p$ is sufficiently small.  We do not know whether
the draw probability is monotone in $p$, nor even whether it is supported on
a single interval (giving a single critical point).

To prove Theorem \ref{thm:hardcoredraw}, we consider a recursion, analogous
to the earlier PCA, expressing game outcomes starting from vertices in $S_k$
in terms of outcomes starting in $S_{k+1}\cup\cdots\cup S_{k+m}$.  Via the
graph isomorphism from $D_k$ to $D$, the iteration of this recursion can be
reinterpreted as a version of Glauber dynamics for the hard-core model on
$D$.  If the hard-core model has multiple Gibbs distributions, then they
correspond to multiple stationary distributions for the recursion on $G$, and
from this we will deduce that draws occur.



Unfortunately, the following very natural example is \emph{not} amenable to
our methods. Let $G$ be the standard cubic lattice $\Z^d$ with edge
orientations given by $\Out(x)=\{x+e_i: 1\leq i\leq d\}$.  Theorem
\ref{thm:hardcoredraw} does not apply for $d\geq3$, because there is no
choice of $m$ and the automorphism $\phi$ such that (A2) holds. We conjecture
that, nonetheless, the trapping game has positive probability of a draw
whenever $p$ is sufficiently small.

\subsection{Further background}

The celebrated \emph{positive rates conjecture} is the assertion that in one
dimension, any finite-state finite-range PCA is ergodic, provided the
transition probability to any state given any neighbourhood states is
positive (the ``positive rates'' condition).
  This contrasts with two and higher
dimensions, where for example
Glauber dynamics for
the low-temperature Ising
model are well known to be non-ergodic.  Despite persuasive
heuristic arguments in favour of the positive rates
conjecture, G{\'a}cs \cite{gacs} has presented an extremely
complicated one-dimensional PCA refuting it. (See also
\cite{gray}.) However, it is still natural to hypothesize
that all ``sufficiently simple'' one-dimensional PCA with
positive rates are ergodic.

The PCA $A_{p,q}$ satisfies the positive rates condition
whenever both $p$ and $q$ are strictly positive.
If $p=0$ or $q=0$, although we no longer have positive rates,
similar but weaker conditions do hold;
$\{0,1\}^n$ has positive probability of yielding any word
in $\{0,1\}^{n-2}$ after \emph{two} steps of the evolution.
In light of this and the above remarks, it would have been
very surprising if these PCA were not ergodic. Nonetheless,
\emph{proving} ergodicity is often very difficult, even in
cases where it appears clear from heuristics or
simulations.

Another case in point is the notorious \emph{noisy
majority} model on $\Z^d$. Here, a configuration is an
element of $\{0,1\}^{\Z^d}$.  The update rule is that with
probability $1-p$, a site adopts the more popular value in
$\{0,1\}$ among itself and its $2d$ neighbours; with
probability $p$ it adopts the other value.  In dimensions
$d\geq 2$ it is expected that this PCA should behave
similarly to the Ising model: it should be ergodic for $p$
sufficiently close to $1/2$, and non-ergodic for $p$
sufficiently small, with a unique critical point separating
the two regimes. However, proving any of this appears very
challenging. See e.g.\ \cite{majority,gray} and the
references therein for more information. One key difficulty
with the noisy majority model is the lack of reversibility
of the
dynamics (in contrast to the Glauber dynamics for the
Ising model, for example). This can be compared to the
difficulty of obtaining a result like
Theorem \ref{thm:hardcoredraw} in the absence
of conditions such as (A1) and (A2); see the discussion
above at the end of Subsection \ref{sec:introhardcore}.



In a different direction, a variant of the notion of envelope cellular
automata has recently been combined with percolation ideas in
\cite{gravner-holroyd}, to prove the surprising fact that certain
deterministic one-dimensional cellular automata exhibit order from
\emph{typical} finitely supported initial conditions, but disorder from
exceptional initial conditions.

\subsection{Organization of the paper}
In Section \ref{sec:main} we explain the link between the
PCA $A_{p,q}$ and the percolation game
in $\Z^2$. We also establish several
basic results concerning monotonicity and
ergodicity. The local weighting on configurations is
introduced in Subsection \ref{sec:weighting}, and the proof
of ergodicity is then given in Subsection
\ref{sec-proof-ergodicity}.

The relation between the trapping game and
the hard-core model is then developed in
Section \ref{sec:hardcore}.  We start by considering the
case of $\Z^2$ where the ideas are simplest,
and in particular we will derive the formula \eqref{win-p} for the winning probability. The case of a more general graph is then treated in Subsection
\ref{sec:hardcore-general}, where Theorem
\ref{thm:hardcoredraw} is proved.
The converse to Theorem \ref{thm:hardcoredraw}
is discussed in Subsection \ref{sec:converse}.
In Subsection
\ref{sec:d3plus} and Theorem \ref{thm:higherd},  we give a
variety of examples of the application of
Theorem~\ref{thm:hardcoredraw} to graphs with vertex set
$\Z^d$ for $d\geq 3$, for which the role of the doubling
graph $D$ is played by various lattice structures. We also
give an extension of Theorem \ref{thm:hardcoredraw} in
Proposition~\ref{prop:extendedhardcoredraw} in Subsection
\ref{sec:hardcore-extension}, using a variant form of the
correspondence to the hard-core model.

We conclude in Section~\ref{sec:open} with some open problems.

\section{Percolation games and probabilistic cellular automata}
\label{sec:main}
\subsection{The PCA for the percolation game}

%
%
%

\label{sec:PCAdef} Consider the percolation game on $\Z^2$ as defined in the
introduction.

\newcommand{\W}{\mathrm{W}}
\renewcommand{\L}{\mathrm{L}}
\newcommand{\D}{\mathrm{D}}

Suppose $x$ is an open site of $\Z^2$. Let $\eta(x)$ be
$\W$, $\L$ or $\D$ according to whether the game started
with the token at $x$ is win for the first player, a loss
for the first player, or a draw, respectively. (Recall that
we assume optimal play, with the players able to see entire
percolation configuration when deciding on
their strategies). If $x$ is a trap, it is
convenient to set $\eta(x)=\W$ (we can imagine that a
player is allowed to move the token to $x$, but with the
effect that the game is then declared an immediate win for
the opponent); similarly if $x$ is a target then we set
$\eta(x)=\L$.

Recall that $\Out(x)=\{x+e_1, x+e_2\}$ is the set of sites to which the token
can move from $x$. By considering the first move, we have the following
recursion for the status of the sites:
\begin{equation}\label{omegarecursion}
\begin{aligned}
\text{$x$ a trap}\;\Rightarrow\;\eta(x)&=\W;\\
\text{$x$ a target}\;\Rightarrow\;\eta(x)&=\L;\\
\text{$x$ open}\;\Rightarrow\;
\eta(x)&=\begin{cases}
\L&\text{if }\eta(y)=\W \text{ for all }y\in \Out(x)\\
\W&\text{if }\eta(y)=\L \text{ for some }y\in \Out(x)\\
\D&\text{otherwise.}
\end{cases}
\end{aligned}
\end{equation}

For $k\in\Z$, let $S_k$ be the set $\{x=(x_1, x_2)\in\Z^2: x_1+x_2=k\}$, a
NW-SE diagonal of $\Z^2$. The recursion (\ref{omegarecursion}) gives us the
values $(\eta(x): x\in S_k)$ in terms of the values $(\eta(x): x\in S_{k+1})$
together with the information about which sites in $S_k$ are traps.

It is important to note that it is not \emph{a priori}
clear whether the recursion \eqref{omegarecursion} suffices
to determine $\eta$ uniquely from the percolation configuration.  Indeed, in the trivial case $p=q=0$ when
all sites are open, we have $\eta(x)=\D$ for all $x$, but
\eqref{omegarecursion} has other solutions: one is to
 set $\eta'(x)$ equal to $\L$ or $\W$ according to
whether $x_1+x_2$ is odd or even.  Such considerations are
in fact central to many of our arguments. One way to
interpret our main result, Theorem~\ref{thm:2d}, is as saying
that \eqref{omegarecursion} does have a unique solution
almost surely whenever $p$ or $q$ is positive.
In contrast, for the higher dimensional
variants considered later, the analogous recursions admit
multiple solutions for certain parameter values.

Via \eqref{omegarecursion}, we can regard the
configurations on successive diagonals $S_k$, as $k$
decreases, as successive states of a one-dimensional PCA.
Let us introduce the following recoding:
 \[\W=0,\quad \L=1,\quad \D=\ques.\]
 (In the coupling arguments below, the symbol $\ques$ will be interpreted as
  marking a site at which the value is ``unknown''.  The choice to assign $\W=0$ and $\L=1$,
rather than the other way around, say, will be important for the later
connection with hard-core models.) The PCA evolves as follows: given the
values for sites in $S_{k+1}$, each value $\eta(x)$ for $x\in S_k$ is derived
independently using the values $\eta(x+e_1)$ and $\eta(x+e_2)$, according to
the scheme given in Figure~\ref{figure-F} (where a $*$ represents an
arbitrary symbol in $\{0,\ques, 1\}$).
%
%
{\begin{figure}
\centering
\begin{tikzpicture}[scale=.6,every text node part/.style={align=left},framed]
\begin{scope}
\node[above] (l) at (-1,1) {\oo};
\node[above] (r) at (1,1) {\oo};
\node[below] (b) at (0,-.6) {\oo\\\ii};
\node[right =2mm of b] (bb) {w.~prob.~$p$\\w.~prob.~$1-p$};
\draw (l.south)--(b.north);
\draw (r.south)--(b.north);
\end{scope}
\begin{scope}[xshift=7cm]
\node[above] (l) at (-1,1) {$*$ \\ \ii};
\node[above] (r) at (1,1) {\ii \\ $*$};
\node[below] (b) at (0,-.6) {\oo \\ \ii};
\draw (l.south)--(b.north);
\draw (r.south)--(b.north);
\node[right =3mm of b] (bb) {w.~prob.~$1-q$\\w.~prob.~$q$};
\end{scope}
\begin{scope}[xshift=15cm]
\node[above] (l) at (-1,1) {{\qq}\\\oo\\ {\qq}};
\node[above] (r) at (1,1) {\oo\\ {\qq}\\ {\qq}};
\node[below] (b) at (0,-.6) {\oo\\{\qq}\\ \ii};
\draw (l.south)--(b.north);
\draw (r.south)--(b.north);
\node[right =3mm of b] (bb) {w.~prob.~$p$\\w.~prob.~$r=1-p-q$\\ w.~prob.~$q$};
\end{scope}
\end{tikzpicture}
\caption{The PCA $F_{p,q}$ ($*$ denotes an arbitrary symbol).}\label{figure-F}
\end{figure}
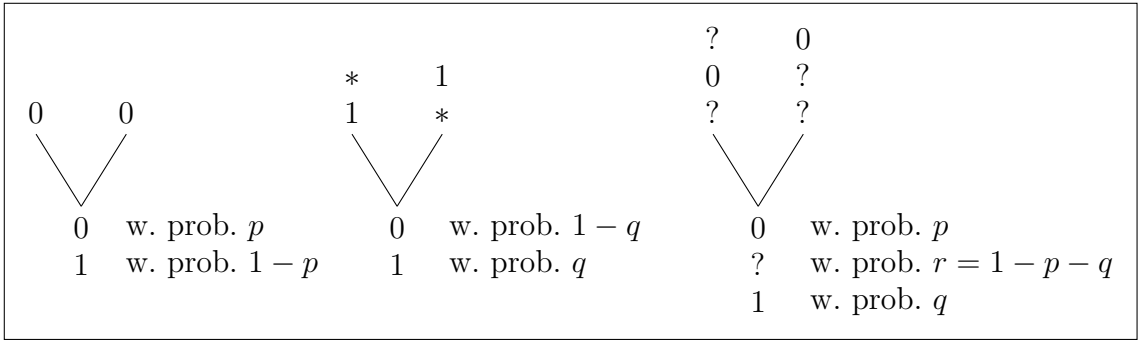
}

We denote the corresponding PCA $F_{p,q}$. Although we have defined it as a
process in the plane, we can also regard it as a PCA on $\Z$ with a
configuration in $\{0,\ques, 1\}^\Z$ evolving in time by setting
\begin{equation}\label{spacetime}
\eta_t(n)=\eta\bigl((-t-n, n)\bigr).
\end{equation}
(Here we have made the arbitrary choice to offset leftward as time
increases,so that the PCA rule gives $\eta_{t+1}(n)$ in terms of
$\eta_{t}(n)$ and $\eta_{t}(n+1)$.) As in Section~\ref{sec:intro2d}, formally
we take $F_{p,q}$ to be an operator on the set of distributions on $\{0,\ques,
1\}^\Z$ representing the action of the PCA.

In the setting of the percolation game, translation invariance of the whole
process on $\Z^2$ implies that the distribution of the configuration on the
diagonal $S_k$ does not depend on $k$; that is, the distribution of
$\big(\eta((k-n, n)): n\in\Z\big)$ does not depend on $k$ and is a stationary
distribution of $F_{p,q}$.  In addition, this distribution is itself invariant under the
action of translations of $\Z$.

We next note two useful monotonicity properties for the PCA
$F_{p,q}$.  In terms of the game, they have natural interpretations: (i) an advantage for one player translates to a disadvantage for the other; and (ii) declaring draws at some positions can only result in more draws elsewhere.
\begin{samepage}
\begin{lem}
\label{lemma:monotonicity} Let $\mu$ and $\nu$ be
probability distributions on $\{0,\ques,1\}^\Z$.
\begin{itemize}
\item[(i)] If $\mu\leq\nu$, where $\leq$ denotes
    stochastic domination with respect to the
    coordinatewise partial order induced by $0<\ques<1$,
    then $F_{p,q}\mu\geq F_{p,q}\nu$.  (Note the reversal of the
    inequality).
\item[(ii)] If $\mu\altpreceq\nu$, where $\altpreceq$
    denotes stochastic domination with respect to the
    coordinatewise partial order induced by
    $0\altless\ques\altgreater 1$, then $F_{p,q}\mu\altpreceq
    F_{p,q}\nu$.
\end{itemize}
\end{lem}
\end{samepage}
\begin{proof}
We can use the recursion (\ref{omegarecursion}) to give a coupling of a
single step of the PCA $F_{p,q}$ started from two different configurations.
Suppose we fix values $(\eta(x): x\in S_{k+1})$ and $(\tilde{\eta}(x): x\in
S_{k+1})$, in such a way that $\eta(x)\leq \tilde{\eta}(x)$ for all $x\in
S_{k+1}$ (where $\leq$ is the coordinatewise order on configurations induced
by $0<\ques<1$). Now use (\ref{omegarecursion}) to obtain values $\eta(x)$
and $\tilde{\eta}(x)$ for $x\in S_k$, using the same realization of
traps, targets, and open sites in $S_k$ in each case. It is straightforward to check that in
that case $\eta(x) \geq \tilde{\eta}(x)$ for each $x\in S_k$. Hence the
operator $F_{p,q}$ is decreasing in the desired sense.

Similarly, if $\eta(x)\altleq\tilde{\eta}(x)$ for all
$x\in S_{k+1}$, then we obtain $\eta(x)\altleq\tilde{\eta}(x)$
also for each $x\in S_k$. So in this case the operator
$F_{p,q}$ is increasing as desired.
\end{proof}

If we restrict the PCA $F_{p,q}$ to configurations that do not contain the symbol
$\ques$, we recover precisely the binary PCA $A_{p,q}$ defined in the introduction.
In the terminology of Bu{\v s}i{\'c} et al.\ \cite{busic_aap}, the PCA $F_{p,q}$
is the \df{envelope} PCA of $A_{p,q}$. A copy of the PCA $F_{p,q}$ can be used to
represent a coupling of two or more copies of the PCA $A_{p,q}$, started from
different initial conditions. The symbol $\ques$ represents a site whose
value is not known, i.e.\ one which may differ between the different copies.

Specifically, consider starting copies of the hard-core PCA $A_{p,q}$ from several
different initial conditions, represented by configurations on the
diagonal $S_k$ for some fixed $k$. As in the proof of
Lemma~\ref{lemma:monotonicity}, a natural coupling is provided by the
recursion (\ref{omegarecursion}), using the same realization of
traps, targets, and open sites in $(S_r: r<k)$.
In particular, let $k>0$ and consider three copies $\eta$, $\tilde{\eta}$
and $\eta^\ques$, with $\eta$ and $\tilde{\eta}$ started from arbitrary
initial conditions on $S_k$, while $\eta^\ques(x)=\ques$ for all $x\in S_k$
(so that $\eta^\ques$ is maximal for the ordering $\altleq$ in Lemma
\ref{lemma:monotonicity}(ii)).  Then we have that $\eta(x)\altleq
\eta^\ques(x)$ and $\tilde{\eta}(x) \altleq \eta^\ques(x)$ for all $x\in S_r$
with $r<k$.  This implies that if $\eta(x)\ne \tilde{\eta}(x)$, then
$\eta^\ques(x)=\ques$.

In terms of the game, we have the following interpretation:
if the origin $O=(0,0)$ is an open site, and
$\eta^\ques(O)=0$ (respectively $\eta^\ques(O)=1$)
then the first (respectively second) player can force
a win within at most $k$ moves of the game.


The ergodicity of an envelope PCA implies the ergodicity of
the original PCA,
but the converse is not true in general. In our case,
however, we can use the monotonicity property in Lemma
\ref{lemma:monotonicity}(i) to show that the two are
equivalent.

\begin{prop}\label{prop-ergodicity}
The PCA $F_{p,q}$ is ergodic if and only if $A_{p,q}$ is ergodic.
\end{prop}

\begin{proof}
It is clear from the definitions that if $F_{p,q}$ is ergodic, then $A_{p,q}$ is also
ergodic. Conversely, suppose that $A_{p,q}$ is ergodic.  Let $\mu$ be a
distribution on $\{0,\ques,1\}^{\Z}$, and let $\delta_{0}$ and $\delta_1$ the
distributions concentrated on the states ``all $0$s'' and ``all
$1$s''.  Then $\delta_0\leq\mu\leq\delta_1$, so by
Lemma~\ref{lemma:monotonicity}(i), for $k\geq 0$ we have either $
F_{p,q}^{k}\delta_{0}\leq F_{p,q}^{k}\mu \leq F_{p,q}^{k}\delta_{1}$ or $
F_{p,q}^{k}\delta_{0}\geq F_{p,q}^{k}\mu \geq F_{p,q}^{k}\delta_{1}$, according to
whether $k$ is even or odd.  But $F_{p,q}^k \delta_{0} =A_{p,q}^k \delta_{0}$ and $
F_{p,q}^k \delta_{1}=A_{p,q}^k\delta_{1}$, and by ergodicity of $A_{p,q}$, the latter two
sequences converge as $k\to\infty$ to the same distribution $\pi$, so
$F_{p,q}^k\mu$ also converges to $\pi$. Thus $F_{p,q}$ is also ergodic.
\end{proof}

\begin{prop}\label{prop-draws}
For each $p$ and $q$, the percolation game has probability
$0$ of a draw if and only if $A_{p,q}$ is ergodic.
\end{prop}

\begin{proof}

If $A_{p,q}$ is ergodic then so is $F_{p,q}$,
and so the unique invariant distribution of $F_{p,q}$
has no $\ques$ symbols. But we know that
the distribution of the game outcomes along a diagonal
$S_k$ is invariant for $F_{p,q}$. Hence with probability 1,
there are no sites from which the game is drawn.

For the converse, let $\omega$ be a random percolation
configuration on $\Z^2$.
Consider any site $x\in S_0$. If the
game started from $x$ is not a draw, then (since at each
turn the player to move has only finitely many options) one
player has a strategy that guarantees a win in fewer than
$N$ moves, where $N\in\N$ is a finite random variable that
depends on $\omega$.  Consequently,
if we assign any configuration of states $0,\ques,1$ to $S_N$ and compute the
resulting states on $(S_n: 0\leq n<N)$ using the recursion
(\ref{omegarecursion}) and the percolation configuration $\omega$,
the resulting state at $x$ is the same as its state for the
percolation game on $\Z^2$ with percolation configuration $\omega$.

Let $\gamma$ be the random configuration of game outcomes
on $S_0$ arising from $\omega$.  Also, fix a distribution
$\nu$ on $\{0,1\}^{\Z}$, and let $\gamma_n$ be the
configuration on $S_0$ that results from assigning a
configuration with law $\nu$ to $S_n$, independent of
$\omega$, and applying (\ref{omegarecursion}) as described
above. By the argument in the previous paragraph, if the
probability of a draw is 0, then $\gamma_n$ converges
almost surely to $\gamma$ (in the product topology).  Hence
also the distribution of $\gamma_n$ converges to that of
$\gamma$.  But $\gamma_n$ has distribution $A_{p,q}^n\nu$, so
$A_{p,q}^n\nu$ converges as $n\to\infty$ to the distribution of
$\gamma$, which does not depend on $\nu$. Hence $A_{p,q}$ is
ergodic.
\end{proof}

\subsection{The weight function}
\label{sec:weighting} We are concerned with the PCA $A_{p,q}$ on
the alphabet $\{0,1\}$, shown in Figure \ref{figure-A},
along with its envelope PCA $F_{p,q}$ shown in Figure \ref{figure-F}.

In order to prove the ergodicity of $F_{p,q}$, we will introduce an appropriate weight on $\ques$ symbols, and prove that this weight decreases under the action of $F_{p,q}$.
The aim of the present section is to motivate the choice of that special weight system.

%
%
%
%
%

%
%
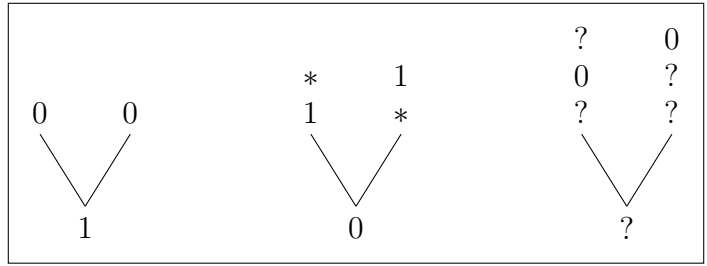
\begin{figure}
\centering
\begin{tikzpicture}[scale=.6,every text node part/.style={align=left},framed]
\begin{scope}
\node[above] (l) at (-1,1) {0};
\node[above] (r) at (1,1) {0};
\node[below] (b) at (0,-.6) {1};
\draw (l.south)--(b.north);
\draw (r.south)--(b.north);
\end{scope}
\begin{scope}[xshift=6cm]
\node[above] (l) at (-1,1) {$*$\\1};
\node[above] (r) at (1,1) {1\\$*$};
\node[below] (b) at (0,-.6) {0};
\draw (l.south)--(b.north);
\draw (r.south)--(b.north);
\end{scope}
\begin{scope}[xshift=12cm]
\node[above] (l) at (-1,1) {{?}\\0\\ {?}};
\node[above] (r) at (1,1) {0\\ {?}\\ {?}};
\node[below] (b) at (0,-.6) {{?}};
\draw (l.south)--(b.north);
\draw (r.south)--(b.north);
\end{scope}
\end{tikzpicture}
\caption{The deterministic cellular automaton $F_{0,0}$.}\label{fig-D}
\end{figure}



We say that the distribution of a configuration $\eta=(\eta_i: i\in \Z)$ is
\df{shift-invariant} if $\eta$ and $(\eta_{i+k}: i\in \Z)$ have the same
distribution for each $k\in \Z$, and \df{reflection-invariant} if $\eta$ and
$(\eta_{-i}: i\in \Z)$ have the same distribution. If $\mu$ is a
distribution and $y\in\{0,\ques,1\}^n$ is a finite word, we write
$\mu(y):=\mu\{\eta:(\eta_1,\ldots,\eta_n)=y\}$ for the corresponding cylinder
probability.

For shift-invariant distributions $\mu$ on
$\{0,\ques,1\}^{\Z}$, we introduce the weight function $w$ defined by
\begin{equation}\label{wdef}
w(\mu)=
\mu(\ques01)+\mu(\ques0)+\mu(\ques)-(p+q)\,\mu(?*1).
\end{equation}
To prove Proposition~\ref{prop:F0} we will establish the inequality
\begin{align}\label{key_ineq}
w(F_{p,q}\mu)\leq w(\mu)-(p+q)\mu(?01),
\end{align}
for any shift-invariant and reflection-invariant distribution $\mu$.
This inequality will indeed ensure that if $\mu$ is $F_{p,q}$-invariant, then $\mu(?01)=0$, which will imply in turn that $\mu(?)=0$.

To give some intuition for the proof of \eqref{key_ineq}, let us focus on the case $p=q=0$. Then the PCA $F_{0,0}$ is in fact deterministic
(see Figure \ref{fig-D}).
Suppose $\mu$ is shift-invariant and reflection-invariant. 
By looking at the possible pre-images of each pattern, we obtain the
following three equalities:
\begin{align*}
F_{0,0}\, \mu (\ques)
&=\mu(\ques\ques)+\mu(0\ques)+\mu(\ques0),\\
F_{0,0}\, \mu (\ques0)
&=\mu(\ques\ques1)+\mu(0\ques1)+\mu(\ques01),\\
F_{0,0}\, \mu (\ques01)&=0.
\end{align*}

Observe that:
$$\mu(\ques\ques)+\mu(\ques0)+\mu(\ques\ques1)+\mu(0\ques1)=\mu(\st\ques\ques)+\mu(\st\ques0)+\mu(\ques\ques1)+\mu(0\ques1)\leq \mu(\ques),$$
where $\st$ represents an unspecified symbol to be summed over.
Using reflection invariance to deduce
$\mu(0\ques)=\mu(\ques0)$, we then obtain that $w(F_{p,q}\mu)\leq w(\mu)$.




In this deterministic case, we can interpret the weight $w(\mu)$
as assigning a weight to each occurrence of the symbol $\ques$
as follows:
\begin{samepage}
\begin{itemize}
\item if a $\ques$ is followed by a $01$,
 then it receives weight $3$; 
\item if a $\ques$ is followed by a $0$ and then by something other than
    a $1$,
it receives weight $2$; 
\item otherwise, a $\ques$ receives weight $1$.
\end{itemize}
\end{samepage}
For a shift-invariant distribution $\mu$,
$w(\mu)$ is then the expected weight per site under $\mu$.

Let us now consider a
symmetric version of the weight system that we have introduced: for each
symbol $\ques$, we add its right-weight, as introduced above, to its
left-weight, which is equal to $3$ if it the previous letter is a $0$ and if
there is a $1$ before it (pattern $10\ques)$, to $2$ if the previous letter
is a $0$ and if there is something else than a $1$ before, and to $1$
otherwise.

Thus, the weight of the symbol $\ques$ in the pattern
$1\ques1$ is equal to $1+1=2$, while in the pattern
$10\ques\ques1$, the weight of the first $\ques$ symbol is
$3$ (left) $+$ $1$ (right) $=4$, and the weight of the
second one is equal to $1+1=2$.

Figure~\ref{fig-evol} shows an example of an evolution of the deterministic CA $F_{0,0}$ from an initial configuration represented at the top (with time going
down the page). The symmetrized weights of the symbols $\ques$ appearing in
the space-time diagram are shown in red. As illustrated in the figure, from a
pattern $1\ques1$, the symbol $\ques$ disappears and the weight thus
decreases, but in other cases the total weight is locally preserved.
Indeed, one can check that starting from any initial configuration
containing finitely many $\ques$ symbols, the total weight
is non-increasing under the action of $F_{0,0}$.

Moving to the general case,
allowing $p$ and $q$ to be positive can be interpreted as introducing
``mutations" into the determinstic evolution described by $F_{0,0}$,
which we control by introducing
the final term into the definition of the weight $w$ in (\ref{wdef}).
\begin{figure}
\begin{center}
\begin{tikzpicture}[thick]
 \matrix[matrix of math
nodes,nodes={circle,fill=white,inner sep=1pt},column
sep={7mm,between origins},row sep={7mm,between origins}]
(m) {
& \i && \q && \i && \i && \o && \q && \q && \i && \o &\\
\o && \o && \o && \o && \o && \q && \q && \o && \o && \i\\
& \i && \i && \i && \i && \q && \q && \q && \i && \o &\\
\o && \o && \o && \o && \o && \q && \q && \o && \o && \i\\
& \i && \i && \i && \i && \q && \q && \q && \i && \o &\\
\o && \o && \o && \o && \o && \q && \q && \o && \o && \i\\
 };
\matrix[yshift=-3.5mm,matrix of math
nodes,nodes={red},column sep={7mm,between origins},row
sep={7mm,between origins}] (l) {
&  && \bf 2 &&  &&  &&  && \bf 4 && \bf 2 && &&  &\\
 &&  &&  &&  &&  && \bf 3 && \bf 3 &&  &&  && \\
&  &&  &&  &&  && \bf 2 && \bf 2 && \bf 2 &&  &&  &\\
&&  &&  &&  &&  && \bf 3 && \bf 3 &&  &&  &&  \\
&  &&  &&  &&  && \bf 2 && \bf 2 && \bf 2 &&  &&  &\\
\  &&  &&  &&  &&  && \bf 3 && \bf 3 &&  &&  && \ \\
 };
\begin{scope}[on background layer]
\clip (m-6-1) rectangle ++(12.5,3.4);
 \draw[shift=(m-2-1),rotate=45,step=0.989949] (-2,-12)
grid (10,2);
\end{scope}
\end{tikzpicture}
\end{center}
\vspace{-10pt}
\caption{Example of evolution of
the weight of a configuration under the operator $F_{0,0}$. Time runs down the page,
 and the weight of each $\ques$ symbol is given below it.}
\label{fig-evol}
\end{figure}
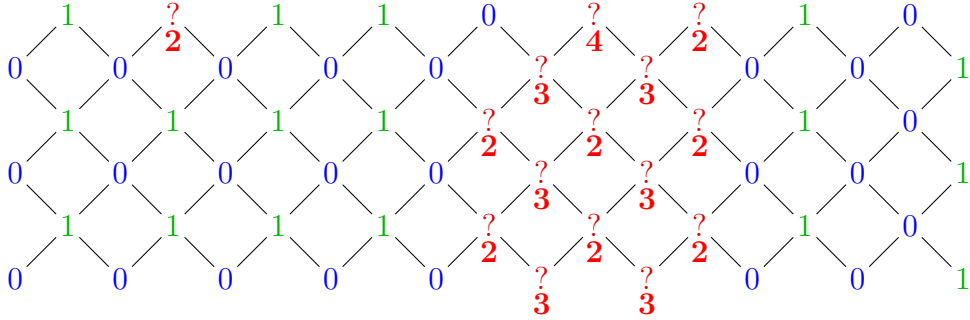

\subsection{Proof of ergodicity}
\label{sec-proof-ergodicity}
\begin{prop} \label{prop:F0}
For any $p$ and $q$, the PCA $F_{p,q}$ has no stationary distribution in which the symbol
$\ques$ appears with positive probability.
\end{prop}
\begin{proof} It suffices to show that there is no shift-invariant and reflection-invariant
stationary distribution in which the symbol $\ques$ appears with positive
probability. For consider iterating the PCA starting from the distribution
$\delta_{\ques}$ concentrated on the configuration with $\ques$ at all sites.
By Lemma \ref{lemma:monotonicity}(ii), the probability
$F_{p,q}^n\delta_{\ques}(\ques)$ is non-increasing, and if there is any
stationary distribution $\mu$ with positive probability of $\ques$, then
$F_{p,q}^n\delta_{\ques}(\ques)$ is bounded below by $\mu(\ques)$ for all $n$,
and so does not converge to 0. Then any limit point of the sequence of
C{\'e}saro sums of $F_{p,q}^n \delta_{\ques}$ is a stationary distribution that
has positive probability of $\ques$, and that is also shift-invariant and
reflection-symmetric.

The idea will be to compare the \emph{weight function} $w$ defined at
(\ref{wdef}) before and after applying a step of the evolution.

To shorten the proof, we write $\tw$ for a pair of consecutive symbols to be summed over the three possibilities $?0,0?,??$ (i.e.\ the pairs which can lead to a $?$ at the next step), so that for example $\mu(\tw1)=\mu(?01)+\mu(0?1)+\mu(??1)$.  We also write $r=1-p-q$.

Suppose $\mu$ is a shift-invariant distribution. Then we have the equalities
\begin{align}
\nonumber
  F_{p,q}\,\mu(?)&=r\mu(\tw);\\
\nonumber
  F_{p,q}\,\mu(?0)&=rp\bigl[\mu(\tw0)+\mu(\tw?)\bigr]+r(1-q)\mu(\tw1);\\
\label{equalities}
  F_{p,q}\,\mu(?01)&=rp(1-p)\mu(\tw00)+rpq\bigl[\mu(\tw0?)
 +\mu(\tw01)+\mu(\tw?)\bigr]  \\
\nonumber
  &\hspace{4.7cm}
+r(1-q)q\mu(\tw1)
 \\
\nonumber
  &=r^2p\mu(\tw00)+rpq\bigl[\mu(\tw0)+\mu(\tw?)\bigr]+r(1-q)q\mu(\tw1).
\end{align}
Summing these three equalities gives
\begin{align*}
&F_{p,q}\,\mu(?)+F_{p,q}\,\mu(?0)+F_{p,q}\,\mu(?01)\\
&=r\mu(\tw)+rp(1+q)\bigl[\mu(\tw0)+\mu(\tw?)\bigr]+r(1-q)(1+q)\mu(\tw1)+r^2p\mu(\tw00)\\
&\leq r\mu(\tw)+(p+q)\bigl[\mu(\tw0)+\mu(\tw?)\bigr]+r\mu(\tw1)+r^2(p+q)\mu(\tw00)
\end{align*}

Suppose in addition that $\mu$ is reflection-invariant. Then
\begin{align*}
\mu(\tw)+\mu(\tw 1)
&\leq \mu(0?)+\mu(?0)+\mu(??)+\mu(?1)+\mu(?01)\\
&\leq \mu(?)+\mu(0?)+\mu(?01)\\
&= \mu(?)+\mu(?0)+\mu(?01).
\end{align*}
Substituting into the previous equality gives:
\begin{align*}
&F_{p,q}\,\mu(?)+F_{p,q}\,\mu(?0)+F_{p,q}\,\mu(?01)\\
&\leq r\bigl[\mu(?)+\mu(?0)+\mu(?01)\bigr]+(p+q)\bigl[\mu(\tw0)+\mu(\tw?)\bigr]+r^2(p+q)\mu(\tw00).
\end{align*}

To deal with the last term on the right, observe that:
$$F_{p,q}\,\mu(?\st1)\geq r(1-p)\mu(\tw00)\geq r^2\mu(\tw00).$$
It follows that:
\begin{align}
\nonumber
w(F_{p,q}\mu)
&=F_{p,q}\,\mu(?)+F_{p,q}\,\mu(?0)+F_{p,q}\,\mu(?01)-(p+q)F_{p,q}\,\mu(?*1)\\
\nonumber
&\leq r\bigl[\mu(?)+\mu(?0)+\mu(?01)\bigr]+(p+q)\bigl[\mu(\tw0)+\mu(\tw?)\bigr]\\
\nonumber
&= \mu(?)+\mu(?0)+\mu(?01)\\
\nonumber
&\hspace{2cm}
+(p+q)\bigl[\mu(\tw0)+\mu(\tw?)- \mu(?)- \mu(?0)- \mu(?01)\bigr]\\
\nonumber
&\leq \mu(?)+\mu(?0)+\mu(?01)-(p+q)\mu(?*1)-(p+q)\mu(?01)\\
\label{wmu}
&=w(\mu)-(p+q)\mu(?01).
\end{align}


The last inequality comes from
\begin{align*}
\mu(\tw0)+\mu(\tw?)+ \mu(?\st1)&\leq \mu(?\st0)+\mu(0?0)+\mu(?\st?)+\mu(0??)+\mu(?\st1)\\
&\leq \mu(?)+\mu(0?).
\end{align*}

Finally suppose that $\mu$ is $F_{p,q}$-invariant. Then since $p+q>0$, it follows
from (\ref{wmu}) that that $\mu(?01)=0$.  This must imply that $\mu(?)=0$.
To explain why, let us first consider the case $p,q>0$. Then, from $0=\mu(?01)=F_{p,q}\,\mu(?01)\geq rpq\,\mu(\tw),$ we obtain $\mu(\tw)=0$, and then $\mu(?)= F_{p,q}\,\mu(?)=r\mu(\tw)=0$.
In the case $p=0$ and $q>0$, using the equations of \eqref{equalities}, we get, successively,
\begin{align*}
0=\mu(?01)=&F_{0,q}\,\mu(?01)=(1-q)^2q\,\mu(\tw1),\\
\mu(?0)=&F_{0,q}\,\mu(?0)=(1-q)^2\mu(\tw1)=0,\\
\mu(?)=&F_{0,q}\,\mu(?)= (1-q)\mu(\tw)=
(1-q)\mu(??)\leq (1-q)\mu(?),
\end{align*}
so that we also deduce that $\mu(?)=0$.
A similar argument applies in the case when $p>0$ and $q=0$.
\end{proof}

Now, we can quickly deduce our main result.

\begin{proof}[Proof of Theorem \ref{thm:2d}]
We know that the distribution of the states (win, loss, draw) of the sites
along a diagonal $S_k$ in the percolation game is a stationary distribution
for $F_{p,q}$. Since by Proposition \ref{prop:F0}, $F_{p,q}$ has no stationary
distribution with positive probability of $\ques$ whenever $p+q>0$, the
probability of a draw in the percolation game must be $0$. Then by Proposition
\ref{prop-draws}, the PCA $A_{p,q}$ is ergodic for each $p$ and $q$ with $p+q>0$.
\end{proof}

\section{Trapping games and the hard-core model}
\label{sec:hardcore}
\subsection{The two-dimensional case}
\label{sec:hardcore2d} In this section we develop the relationship between
the trapping game and the hard-core model.  We start in the setting of
$\Z^2$ where the ideas are easiest to understand, but our main application
will be in Section \ref{sec:hardcore-general}, when we establish a more
general framework, and apply it to show that certain higher-dimensional games
have positive probability of a draw when $p$ is sufficiently small.

Consider the hard-core PCA $A_{p,0}$. This PCA is known to belong to a family of
one-dimensional PCA having a stationary distribution that is itself a
stationary Markov chain indexed by $\Z$~\cite{beljaev, dobrushin,
mairesse_ihp}.  This distribution, $\mu_{p}$ say, is the law of the
stationary Markov chain on $\Z$ with transition matrix
\begin{equation}\label{matrix}
P=\begin{pmatrix}
p_{0,0}&p_{0,1}\\
p_{1,0}&p_{1,1}\end{pmatrix}
=
\begingroup
\renewcommand*{\arraystretch}{2}
\begin{pmatrix}
 { \frac{2-p-\sqrt{p(4-3p)}}{2(1-p)^2}}
 &{\frac{2p^2-3p+\sqrt{p(4-3p)}}{2(1-p)^2}}\\
{\frac{-p+\sqrt{p(4-3p)}}{2(1-p)}}
&{\frac{2-p-\sqrt{p(4-3p)}}{2(1-p)}}
\end{pmatrix}
\endgroup,
\end{equation}
on state space $\{0,1\}$. (See Section 4.2 of \cite{mairesse_tcs} -- note
that $p$ there corresponds to our $1-p$). In fact, the evolution of the PCA
started from $\mu_p$ is time-reversible -- the distribution of the
two-dimensional space-time diagram obtained (via the correspondence at
(\ref{spacetime})) is invariant under reflection in the line $x_1+x_2=k$ for
any $k$. (In addition, the distribution $\mu_p$ is itself reversible as a
Markov chain on $\Z$, which corresponds to symmetry of the two-dimensional
picture under reflection in the line $x_1=x_2$).

By Theorem~\ref{thm:2d}, we know that $\mu_p$ is in fact the unique
stationary distribution of $F_{p,0}$. Therefore the probability
that either the first player wins the trapping game starting from
the origin, or the origin is a trap, is
$$\mu_p(0)=\frac{p_{1,0}}{p_{1,0}+p_{0,1}}
=\frac12\biggl(1+\sqrt{\frac{p}{4-3p}}\biggr).$$
By conditioning on the event that the origin is open, we then find
that the probability that the win probability is
$(\mu_p(0)-p)/(1-p)$ which corresponds to the
quantity given in \eqref{win-p}

An illuminating way to understand the presence of this Markovian reversible
stationary distribution is to consider the \emph{doubling graph} of the PCA,
corresponding to two consecutive times of its evolution \cite{vasilyev,
kozlov, dobrushin}. This is an undirected bipartite graph, connecting sites
between which there is an influence induced by the rules of the PCA.

As in Section \ref{sec:PCAdef}, we can think of
a configuration of the PCA as indexed by a diagonal
$S_k=\{(x_1,x_2): x_1+x_2=k\}$ of $\Z^2$.
A time-step of the PCA then corresponds to moving
from a configuration on $S_{k+1}$ to a configuration on $S_k$.

As before, let $\Out(x)=\{x+e_1, x+e_2\}$ for $x\in S_k$. The elements of
$\Out(x)$ lie in $S_{k+1}$, and are the sites to which the token may move
from sites $x$; they are the sites whose values appear on the right side of
the recurrence (\ref{omegarecursion}) for the value $\eta(x)$. Then the
bijection $\phi:\Z^2\to\Z^2$ given by
\begin{equation}\label{phidef}
\phi(x)=x+e_1+e_2,
\end{equation}
which maps $S_k$ to $S_{k+2}$ for each $k$,
has the following symmetry property: for all $x$ and $y$,
\begin{equation}
\label{doublingsymmetry}
y\in \Out(x) \quad\text{if and only if}\quad \phi(x)\in \Out(y).
\end{equation}

Let $D_k$ be the undirected bipartite graph with vertex set $S_k\cup
S_{k+1}$, and an edge joining $x\in S_k$ and $y\in S_{k+1}$ if $y\in
\Out(x)$.

The graphs $D_k$ are isomorphic to each other for all $k\in\Z$.
The \df{doubling graph} is a generic graph $D$ that is isomorphic to each $D_k$.
We can also interpret $D$ as the image of $\Z^2$ under the
equivalence relation $x\equiv \phi(x)$.
More simply, we can take $D$ to be $\Z$ with nearest-neighbour edges, as shown in Figure
\ref{fig:doublinggraph}. Consider the map $v:\Z^2\to\Z$ given by
\begin{equation}\label{isom}
v\big( (x_1, x_2) \big) = x_1-x_2.
\end{equation}
Restricted to the set $S_k\cup S_{k+1}$, this gives an isomorphism between
$D_k$ and $D$, for any $k$.


\begin{figure}
\begin{center}
\includegraphics[scale=0.9]{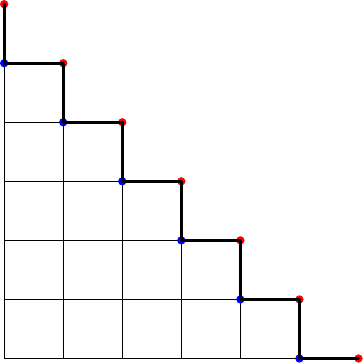}
\includegraphics[scale=0.9]{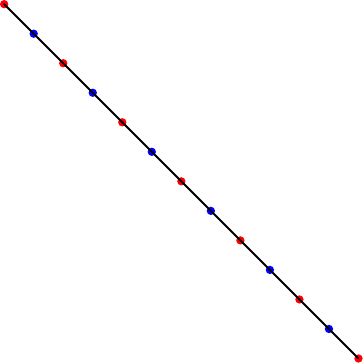}
\end{center}
\caption{
The doubling graph $D$, isomorphic to $\Z$, shown on the left
in correspondence with two successive diagonals $S_k$, $S_{k+1}$
of $\Z^2$.}
\label{fig:doublinggraph}
\end{figure}

Recall the definition of the hard-core model as given in Section
\ref{sec:introhardcore}; a Gibbs distribution for the hard-core model on a
graph with vertex set $W$ with activity $\lambda>0$ is a distribution on
configurations $\eta\in\{0,1\}^W$
satisfying (\ref{hardcoredef}).

Consider the hard-core model on the doubling graph $D$ with
vertex set $W=\Z$. This is a bipartite graph, with
bipartition $W=W_0\cup W_1$ where $W_0$ and $W_1$ are the
sets of even and odd integers respectively. We consider the
following two update procedures for configurations on
$\{0,1\}^W$. For an ``odd" update, for each vertex $x\in
W_1$ independently, resample $\eta(x)$ according to the
values at its two neighbours, setting $\eta(x)=0$ with
probability 1 if either of the neighbours takes value 1,
and otherwise setting $\eta(x)=1$ with probability $1-p$.
For an ``even" update, do the same for vertices in $W_0$.
Set $\lambda=1/p-1$, so that $1-p=\lambda/(1+\lambda)$.
Since each of $W_0$ and $W_1$ is an independent set of $D$,
any Gibbs distribution for the hard-core model with
activity $\lambda$ is invariant under both of these update
operations. (This is a version of Glauber dynamics).

Take some even $k\in\Z$. Suppose we start from a configuration on
$\{0,1\}^W$, which, via the isomorphism (\ref{isom}) between $D$ and $D_k$
under which $W_0$ maps to $S_k$ and $W_1$ to $S_{k+1}$, corresponds to a
configuration in $\{0,1\}^{S_k\cup S_{k+1}}$.  Perform an odd update,
resampling the sites of $W_1$, leading to a new configuration on $\{0,1\}^W$.
Considering now (\ref{isom}) as an isomorphism between $D_k$ and $D_{k-1}$,
which maps $W_0$ to $S_k$ and $W_1$ to $S_{k-1}$, the updated configuration
on $\{0,1\}^W$ corresponds to a configuration in $\{0,1\}^{S_{k-1}\cup S_k}$,
whose values at the sites in $S_k$ are left unchanged.  We can interpret the
update as generating a configuration on $S_{k-1}$ from a configuration on
$S_k$.  This procedure is identical to that which occurs in one iteration of
the PCA $A_{p,0}$.

If we then perform an even update, resampling the sites of $W_0$,
we can pass in the same way to a configuration on
the sites of $S_{k-2}\cup S_{k-1}$, which corresponds
to the next step of the PCA.

Continuing to perform odd and even updates alternately, we reproduce the
evolution of the PCA. A Gibbs distribution on $D$ is characterized by its
marginal on the vertices of one half of the bipartition, say $W_0$. Since the
distribution is preserved by the updates, this distribution on
$\{0,1\}^{W_0}$ is $2$-periodic for the PCA.

In fact, for any $\lambda$ there is a unique Gibbs distribution for the
hard-core model on $\Z$.  Since the hard-core interaction is homogeneous and
nearest-neighbour, this Gibbs distribution is itself a stationary Markov
chain indexed by $\Z$.  Let $Q=Q_p$ be its transition matrix.  Therefore, the
marginal distributions on $W_0$ and $W_1$ are in fact equal to each other.
Call this marginal distribution $\mu_p$.  Then $\mu_p$ is the law of the
stationary Markov chain with transition matrix $P=Q^2$.  This $\mu_p$ is a
stationary distribution for the PCA $A_{p,0}$, and the matrix $P$ is the one in
\eqref{matrix}. See Figure \ref{fig:twomu} for an illustration.
\begin{figure}
\begin{center}
\begin{tikzpicture}[>=latex]
  \draw[dashed]  (0,0)node[left=3pt]{$\mu_p$}--(8,0);
  \draw[dashed] (0,1.5)node[left=3pt]{$\mu_p$}--(8,1.5);
  \draw[thick] (-1,0)edge[bend left,->] node[left]{$A_{p,0}$} (-1,1.5);
  \draw[thick] (8.5,1.5)edge[bend left,->] node[right]{$A_{p,0}$} (8.5,0);
  \node[circle,fill=red,inner sep=1.5pt] at (0,0) (0){};
  \node[circle,fill=red,inner sep=1.5pt] at (2,0) (2){};
  \node[circle,fill=red,inner sep=1.5pt] at (4,0) (4){};
  \node[circle,fill=red,inner sep=1.5pt] at (6,0) (6){};
  \node[circle,fill=red,inner sep=1.5pt] at (8,0) (8){};
  \node[circle,fill=blue,inner sep=1.5pt] at (1,1.5) (1){};
  \node[circle,fill=blue,inner sep=1.5pt] at (3,1.5) (3){};
  \node[circle,fill=blue,inner sep=1.5pt] at (5,1.5) (5){};
  \node[circle,fill=blue,inner sep=1.5pt] at (7,1.5) (7){};
  \draw (0)--(1)--(2)--(3)--(4)--(5)--(6)--(7)--(8);
\end{tikzpicture}
\end{center}
\caption{
The Markovian distribution corresponding to a Gibbs measure
for the hard-core model on the doubling graph $W$
yields a Markovian distribution $\mu_p$ on
each of the two vertex classes $W_0$ and $W_1$.
Since the Gibbs distribution is invariant
under the update procedures, the distribution $\mu_p$
is invariant for the PCA.}
\label{fig:twomu}
\end{figure}
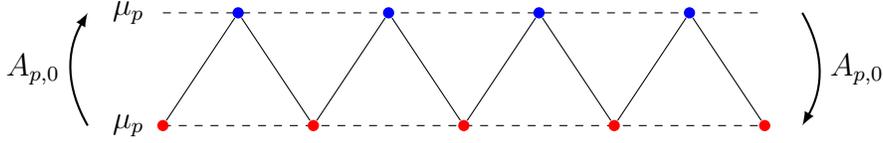


Combining with monotonicity properties (of the sort written
in Lemma \ref{lemma:monotonicity}(i)), one can use the uniqueness of the Gibbs measure on $\Z$
to conclude that $A_{p,0}$ indeed has a unique stationary distribution;
we explain this in a more general setting in Section \ref{sec:converse}.

Before that, in the next section we will use the implication in the other direction to
establish the result of Theorem \ref{thm:hardcoredraw}; in situations where there exist multiple Gibbs distributions for
the hard-core model, we can conclude that there are multiple periodic
distributions for the corresponding PCA; then the PCA is non-ergodic, and
draws occur with positive probability in the corresponding game.

\subsection{General framework}
\label{sec:hardcore-general} Recall that in the setting of Theorem
\ref{thm:hardcoredraw}, we have a locally finite graph $G$ with vertex set
$V$, along with a partition $(S_k: k\in \Z)$ of $V$ and an integer $m\geq 2$,
such that conditions (A1) and (A2) given in Section \ref{sec:introhardcore}
hold.

We also defined $D_k$ be the graph with vertex set $S_k\cup\dots\cup
S_{k+m-1}$, with an undirected edge $(x,y)$ whenever $(x,y)$ is a (directed)
edge of $V$. For convenience we will also use $D_k$ to denote the vertex set
$S_k\cup\dots\cup S_{k+m-1}$.

\begin{lem}\label{lemma:Dk}
The graphs $D_k$ are isomorphic to each other for all $k\in\Z$.
\end{lem}

\begin{proof}
Consider the map $\chi_k$ defined on $D_k$
under which
\[
\chi_k(x)=\begin{cases}
x&\text{ if } x\in S_{k+1}\cup\dots\cup S_{k+m-1}\\
\phi(x)&\text{ if } x\in S_k
\end{cases}.
\]
From assumptions (A1) and (A2) above,
$\chi_k$
is a graph isomorphism from $D_k$ to $D_{k+1}$.
Hence indeed $D_k$ and $D_{k+1}$ are isomorphic,
and so by induction any two $D_k$, $D_{k'}$ are isomorphic.
\end{proof}




We then take $D$
to be a graph isomorphic to any $D_k$.
(When $m=2$ we sometimes call $D$ the \df{doubling graph}).  Note that $D$ is $m$-partite.
Specifically, let us fix some isomorphism $f_0$ from $D_0$ to $D$, and let
$W_i$ be the image of $S_i$ under $f_0$, for $i=0,\dots, m-1$. Then
$(W_0,\ldots,W_{m-1})$ is a partition of the vertices of $D$ into $m$
classes, and assumption (A1) guarantees that there are no edges within a
class $W_i$.

It will be important that we can map both $D_k$ and $D_{k+1}$ to $D$ in such
a way that the vertices common to $D_k$ and $D_{k+1}$ have the same image in
both maps.
\begin{samepage}
\begin{lem}\label{lemma:DkDk+1}
There exists a family of maps $(f_k: k\in\Z)$ such that $f_k$ is a graph
isomorphism from $D_k$ to $D$, and such that the following properties hold.
\begin{itemize}
\item[(a)] For each $k$,
\[
f_k(x)=\begin{cases}
f_{k+1}(x)
&\text{ for } x\in D_k\cap D_{k+1}=S_{k+1}\cup\dots\cup S_{k+m-1}\\
f_{k+1}(\phi(x))
&\text{ for } x\in S_k.
\end{cases}
\]
\item[(b)] For each $k$ and each $r\in\{k,k+1,\dots,k+m-1\}$, the image
    of $S_r$ under $f_k$ is $W_{r\newmod m}$.
\item[(c)] Let $x\in S_k$ and $y\in D_k=S_k\cup\dots\cup S_{k+m-1}$. Then
    $y\in \Out(x)$ if and only if $f_k(y)$ is a neighbour of $f_k(x)$ in
    $D$.
\end{itemize}
\end{lem}
\end{samepage}

\begin{proof}
Let $f_0$ be the isomorphism from $D_0$ to $D$ described just above.
Then we can compose $f$ with the isomorphisms $\chi_k$ defined in
the proof of Lemma \ref{lemma:Dk}, by setting
\[
f_k=
\begin{cases}
f_0\circ\chi_{-1}\circ\chi_{-2}\circ\dots\circ\chi_k
&\text{ for }k<0\\
f_0\circ\chi_0^{-1}\circ\chi_1^{-1}\circ\dots\circ\chi_{k-1}^{-1}
&\text{ for }k>0.
\end{cases}
\]
Then using assumption (A2), it is easy to check by induction upwards and
downwards from $0$ that $f_k$ is an isomorphism from $D_k$ to $D$ satisfying
the properties stated in (a) and (b), for each $k$.

Finally note that by (A1), if $x\in S_k$ then $\Out(x)\subseteq
S_k\cup\dots\cup S_{k+m-1}$ while $\In(x)$ is disjoint from
$S_k\cup\cdots\cup S_{k+m-1}$. By definition, the set of neighbours of $x$ in
the graph $D_k$ is then $\Out(x)$. Then part (c) follows since $f_k$ is a
graph isomorphism from $D_k$ to $D$.
\end{proof}

Given a hard-core configuration in $\{0,1\}^D$, we can consider Glauber
update steps that resample the vertices of one of the vertex classes $W_0$,
$W_1,\dots,W_{m-1}$. To perform an update of the class $W_i$: for each $v\in
W_i$ independently, let the new value at $v$ be 0 if any neighbour of $x$ has
value 1, and otherwise let the new value at $v$ be 0 with probability
$p=1/(1+\lambda)$ and 1 with probability $1-p=\lambda/(1+\lambda)$. If a
distribution on $\{0,1\}^D$ is a Gibbs distribution for the hard-core model
on $D$ with activity $\lambda=1/p-1$, then it is invariant under this update
procedure for each $i=0,1,\dots, m-1$. (Again, this is a version of the
Glauber dynamics for the hard-core model on $D$.)

\begin{proof}[Proof of Theorem \ref{thm:hardcoredraw}]
We start by defining an analogue of the hard-core PCA $A_{p,0}$ in the general setting. As
before, we have the recursion (\ref{omegarecursion}) for the outcome of the
game started from $x\in V$, in terms of the outcomes started from the
elements of $\Out(x)$ together with the information whether $x$ itself
is a trap or open. (Recall that we treat the game from $x$ as a win if $x$ is
a trap.)

As in previous sections we can specialize that recursion to configurations
involving only the symbols $0=W$ and $1=L$. This gives the following
recursion for a family of variables $(\gamma(x): x\in V)\in\{0,1\}^V$ (which
we do not assume to be necessarily game outcomes):
\begin{equation}
\begin{aligned}
\text{$x$ a trap}\;\Rightarrow\;\gamma(x)&=0;\\
\text{$x$ open}\;\Rightarrow\;
\gamma(x)&=\begin{cases}
1&\text{if }\gamma(y)=0 \text{ for all }y\in \Out(x)\\
0&\text{if }\gamma(y)=1 \text{ for some }y\in \Out(x).
\end{cases}
\end{aligned}
\label{01recursion}
\end{equation}

If $x\in S_k$, then $\Out(x)\subset S_{k+1}\cup\dots\cup S_{k+m-1}$. Thus,
the recursion (\ref{01recursion}) gives $(\gamma(x): x\in S_k)$ in terms of
$(\gamma(x): x\in S_{k+1}\cup\dots\cup S_{k+m-1})$ and the random
percolation configuration of traps and open sites in $S_k$ (which we take as usual to be
product measure with each site being a trap with probability $p$). This is
analogous to the PCA $A_{p,0}$ considered earlier (although for $m>2$, a ``state"
of the PCA is now more complicated to describe).

Fix $K\in\Z$, and take some boundary condition $(\gamma(x): x\in
S_K\cup\cdots\cup S_{K+m-1})$, which we allow to be be random, but which is
independent of the percolation configuration in $\bigcup_{r<K}
S_r$.
Applying (\ref{01recursion}) repeatedly then generates an evolution
$(\gamma(x): x\in S_r)_{r\leq K+m-1}$.

We will couple this evolution with a process of configurations in
$\{0,1\}^D$. For $k\leq K$ and $v\in D$, define
$\sigma_k(v)=\gamma(f_k^{-1}(v))$.  Then $\sigma_k\in\{0,1\}^D$ for each $k$.
The idea is now to show that the transformation from $\sigma_{k+1}$ to
$\sigma_{k}$ is identical to a hard-core update of the vertex class
$W_{k\newmod m}$, with randomness provided by the percolation configuration of open in $S_k$.  Notice that $\sigma_{k+1}$ is a function of
$(\gamma(x):x\in S_{k+1}\cup\dots\cup S_{k+m})$, while $\sigma_{k}$ is a
function of $(\gamma(x): x\in S_{k}\cup\dots\cup S_{k+m-1})$.

If $v\in W_i$ where $i\ne k\newmod m$, then by Lemma \ref{lemma:DkDk+1}(a)
and (b), $f_{k+1}^{-1}(v)=f_k^{-1}(v)\in S_{k+1}\cup \dots \cup S_{k+m-1}$.
Thus $\sigma_{k+1}(v)=\sigma_k(v)$. So the only sites in $D$ which can change
their value between the configuration $\sigma_{k+1}$ and the configuration
$\sigma_k$ are those in $W_{k\newmod m}$.  Consider such a $v\in W_{k\newmod
m}$, and let $x=f^{-1}_k(v)$ so that $x\in S_k$ (by Lemma
\ref{lemma:DkDk+1}(b)) and $\sigma_k(v)=\gamma(x)$.

Translating (\ref{01recursion}) and using Lemma \ref{lemma:DkDk+1}(c) gives
that for $v\in W_{k\newmod m}$:
\begin{equation}
\begin{aligned}
\text{$x$ a trap}\;\Rightarrow\;\sigma_k(v)&=0;\\
\text{$x$ open}\;\Rightarrow\;
\sigma_k(v)&=\begin{cases}
1&\text{if }\sigma_{k+1}(u)=0 \text{ for all }u\sim v\\
0&\text{if }\sigma_{k+1}(u)=1 \text{ for some }u\sim v.
\end{cases}
\end{aligned}
\label{sigmarecursion}
\end{equation}

Each bit of randomness (the information about whether $x\in S_k$ is an open site or a trap) is used only once. Since each $x$ is a trap with probability $p$ independently,
we have that the conditional distribution of $\sigma_k$
given $\sigma_{k+1}, \sigma_{k+2},\dots, \sigma_K$ is precisely that
obtained by performing a hard-core update of the vertex class $W_{k\newmod m}$.

Now let $\mu$ be a Gibbs distribution for the hard-core model on $D$.  By
choosing the distribution of the boundary condition $(\gamma(x): x\in
S_K\cup\cdots\cup S_{K+m-1})$ correspondingly, we can arrange that $\sigma_K$
has distribution $\mu$. But then since $\mu$ is invariant under the hard-core
updates, $\sigma_k$ has distribution $\mu$ for all $k<K$ also.

So, suppose that there exist multiple Gibbs distributions, and let $\mu$ and
$\nu$ be two of them.  By alternating between $\mu$ and $\nu$, we can arrange
a sequence indexed by $K$ of boundary conditions $(\gamma^{(K)}(x): x\in
S_K\cup\cdots\cup S_{K+m-1})$ that induces a sequence of distributions of
$\sigma^{(K)}_0$ having both $\mu$ and $\nu$ as limit points as $K\to\infty$.
In particular, the configuration $\sigma^{(K)}_0$ does not converge almost
surely as $K\to\infty$ (in the product topology).  So the sequence of
configurations $(\gamma^{(K)}(x): x\in S_0\cup\cdots\cup S_{m-1})$ does not
converge almost surely.

Now we apply the same argument that we used for the second part of the proof
of Proposition \ref{prop-draws}. If the game started from $x$ is not a draw,
then one player has a strategy which guarantees a win in fewer than $N$
moves, where $N\in\N$ is an almost surely finite random variable which
depends on the configuration.  Then the values $\gamma^{(K)}(x)$
must agree for all large enough $K$.  Hence if there is zero probability of a
draw from each site, then the configuration $(\gamma^{(K)}(x): x\in
S_0\cup\cdots\cup S_{m-1})$ converges almost surely as $K\to\infty$. By the
argument in the previous paragraph, this contradicts the existence of
multiple hard-core Gibbs distributions.
\end{proof}

\subsection{Remarks on the converse direction of Theorem \ref{thm:hardcoredraw}}
\label{sec:converse} We do not know whether the converse of Theorem
\ref{thm:hardcoredraw} holds in general; that is, whether the uniqueness of the
hard-core Gibbs distribution on $D$ implies that there are no draws
for the game on $G$. However, for the case $m=2$ (when the graph $D$ is bipartite), it is indeed the case that this converse direction also holds.

This can be established using a relatively standard argument involving
Gibbs measures. Let us outline the argument. Let $m=2$, and let $D$
have vertex set $W=W_0\cup W_1$ as above. Let us write $B$ for the operator
which first does a hard-core update of the sites in $W_0$, and then a hard-core
update of the sites in $W_1$.

\newcommand{\ctop}{\sigma^{\max}}
\newcommand{\cbottom}{\sigma^{\min}}

We can define a partial order on configurations in $\{0,1\}^W$
by setting $\sigma\leq\tsigma$ if $\sigma(v)\leq \tsigma(v)$ for
all $v\in W_0$ and $\sigma(v)\geq \tsigma(v)$ for all $v\in W_1$.
Then the operator $B$ (acting on distributions on $\{0,1\}^W$)
is monotonic; if $\mu\leq\nu$ in the sense of stochastic domination,
then also $B\mu\leq B\nu$.
(This is closely related to the property in Lemma \ref{lemma:monotonicity}(i).)
Let us write
$\ctop$ for the maximal configuration (taking value $1$ on all sites of
$W_0$ and $0$ on all sites of $W_1$)
and $\cbottom$ for the minimal configuration.

Then a monotonicity argument similar to that in Proposition \ref{prop-ergodicity}
and Proposition \ref{prop-draws} establishes that the probability of
a draw is $0$ precisely if there is a unique stationary distribution on $\{0,1\}^W$ for the operator $B$.

Suppose therefore that there is a unique Gibbs measure $\mu$ for the hard-core
model on $D$. We wish to show that $\mu$ is the unique invariant
distribution for the operator $B$.

Consider any increasing sequence of finite subsets of $W$,
$\Lambda_1\subset\Lambda_2\subset\dots\subset\Lambda_n\subset\dots$,
with $\bigcup \Lambda_n=W$. Let $\mu^{\max}_{k,n}$ be the distribution
obtained by applying the hard-core update procedure $k$ times
on the sites of $\Lambda_n$, with the values at sites of $\Lambda_n^c$
fixed, starting from the configuration $\ctop$.
Let $\mu^{{\min}}_{k,n}$ be the analogous distribution starting
from $\cbottom$.

We have that $\mu^{\max}_{k,n}\downarrow \mu^{\max}_n$
and $\mu^{\min}_{k,n}\uparrow \mu^{\min}_n$
as $k\to\infty$,
where $\mu^{\max}_n$ and $\mu^{\min}_n$ are the hard-core
measures on the finite vertex set $\Lambda_n$
with maximal and minimal boundary conditions respectively.

Furthermore, $\mu^{\max}_n \to \mu$ and $\mu^{\min}_n \to \mu$
as $n\to\infty$ (since any limit point of a sequence of
Gibbs measures on the subsets $\Lambda_n$ is a Gibbs measure on
the entire vertex set $W$, and we assume that $\mu$ is the only such
Gibbs measure on $W$).

But if $\nu$ is any measure on $\{0,1\}^W$, then
from the monotonicity of $B$ we have that $\mu^{\min}_{k,n}\leq
B^k \nu\leq \mu^{\max}_{k,n}$ for all $k$.
In particular suppose that $\nu$ is invariant. Then $B^k\nu=\nu$ for all $k$.
But then by taking $k$ and $n$ sufficiently large, we can sandwich
$\nu$ between two measures which are as close to $\mu$ as desired.
This gives that $\nu=\mu$, and so indeed $B$ has a unique
stationary distribution. Hence the probability of a draw is $0$,
as required.

If however $m>2$, the graph is no longer bipartite, and the monotonicity
argument above no longer works. In particular, we cannot define a partial
ordering on configurations in the same way. Such a partial ordering
was already used in the proof of Proposition \ref{prop-ergodicity},
and without monotonicity, it is no longer clear that uniqueness of the invariant distribution for the PCA-like evolution on $\{0,1\}^W$ implies that the
game has probability $0$ of a draw.
(Closely related examples in which ergodicity of a binary PCA does not
imply ergodicity of its envelope PCA are noted in \cite{busic_aap}.)
Furthermore, without monotonicity we could
no longer apply the argument above involving bounds in terms of
Gibbs measures on finite subsets with maximal and minimal boundary conditions.

\subsection{Example graphs with $d\geq 3$}
\label{sec:d3plus}  We now give several examples of graphs $G$ to which one
may hope to apply Theorem~\ref{thm:hardcoredraw}.  (As we will see, the
result can indeed be applied in some cases, but not in others.)  We consider
the vertex set $V=\Z^d$ (or subsets thereof), for various different choices
of the set $\Out(x)$ of vertices to which the token can move from $x$. In
each case, we consider the trapping game in which each site is a trap with
probability $p$ and open with probability $1-p$. All our examples can be
regarded as natural extensions of the original $\Z^2$ game, in the sense that
they reduce to it when we set $d=2$.

\begin{exa}\label{game1}
Let $\Out(x)=\{x+e_i: 1\leq i\leq d\}$. So $|\Out(x)|=d$. This is perhaps the
most natural extension of all. However we cannot apply Theorem
\ref{thm:hardcoredraw} because there is no choice of the automorphism $\phi$
for which assumption (A2) holds.
\end{exa}

\begin{exa}\label{game2} Let
$\Out(x)=\{x\pm e_i + e_d: 1\leq i\leq d-1\}$. This is the example already
mentioned in Section \ref{sec:introhardcore}. Here $|\Out(x)|=2(d-1)$. Since
any step preserves parity, it is natural to restrict to the set of even sites
$\Z^d_{\text{even}}:=\{x\in\Z^d: \sum x_i \text{ is even}\}$.

In two dimensions, the game is isomorphic to the original game on $\Z^2$. For
general $d$, conditions (A1) and (A2) hold with $m=2$ if we set
$S_k=\{x\in\Z^d_\text{even}:x_d=k\}$ and $\phi(x)=x+2e_d$.

To obtain the doubling graph,
consider $D_k=S_k\cup S_{k+1}$ with an edge
between $x\in S_k$ and $y\in S_{k+1}$
whenever $y\in\Out(x)$.
This gives a graph isomorphic to the standard
cubic lattice $\Z^{d-1}$
(for example, the map
\[
(x_1,\dots, x_{d-1},x_d)\to(x_1,\dots, x_{d-1})
\]
gives a graph isomorphism from $D_0=S_0\cup S_1$ to $\Z^{d-1})$.

The graph is vertex-transitive. By Theorem \ref{thm:hardcoredraw}, if there
exist multiple Gibbs distributions for the hard-core model on $\Z^{d-1}$ with
activity $\lambda$, then the percolation game on $G$ with $p=1/(1+\lambda)$
has positive probability of a draw from any vertex.
\end{exa}

\begin{exa}\label{game3}
Now let $\Out(x)=\{x\pm e_1 \pm e_2\dots\pm e_{d-1} +
e_d\}$, so that $|\Out(x)|=2^{d-1}$. Each step changes the
parity of every coordinate, so we restrict to the set
$\Z^d_{\text{bcc}} =\{x\in\Z^d: x_i\equiv x_j\bmod 2 \text{
for all }i,j\}$. Putting an edge between $x$ and $y$
whenever $y\in \Out(x)$, we obtain the \df{body-centred
cubic lattice} in $d$ dimensions. This consists of two
copies of $(2\Z)^d$, each offset from the other by
$(1,1,\dots, 1)$, so that each point of one lies at the
centre of a unit cube of the other; the edges are given by
joining each point to the $2^{d}$ corners of the
surrounding unit cube. See Figure \ref{fig:bcc}
for an illustration.
\begin{figure}[!]
\begin{center}
\includegraphics[width=0.55\textwidth]{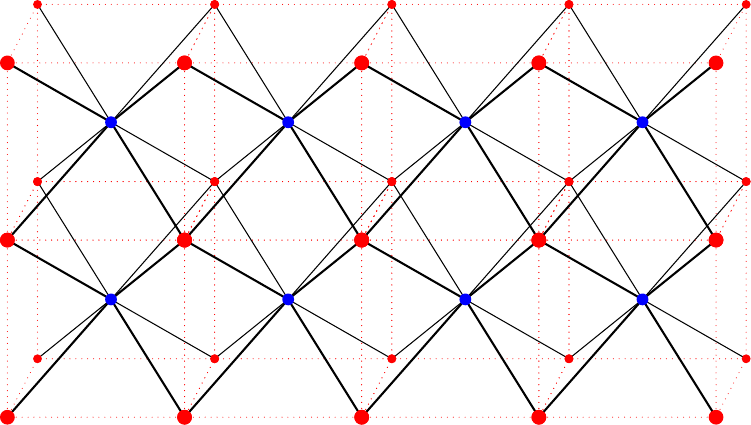}
\end{center}
\caption{The three-dimensional
body-centred cubic lattice
(which is the doubling graph for the PCA
associated to Example \ref{game3} when $d=4$).
The two underlying copies of $\Z^3$ are shown with red
and with blue vertices. The black lines are the edges
of the body-centred cubic lattice, and the dotted red lines
show the nearest-neighbour edges in the red copy of $\Z^3$.
}\label{fig:bcc}
\end{figure}

Conditions (A1) and (A2) hold for $m=2$ with $S_k=\{x\in\Z^d_{\text{bcc}}:
x_d=k\}$ and $\phi(x)=x+2e_d$.  The doubling graph $D$ isomorphic to
$D_k=S_k\cup S_{k+1}$ for each $k$ is now the body-centered cubic lattice in
$d-1$ dimensions.  The map $v(x)=(x_1, x_2, \dots, x_{d-1})$ from
$\Z^d_{\text{bcc}}$ to $\Z^{d-1}_{\text{bcc}}$ restricts to an isomorphism
between $D_k$ and $D$ for each $k$.

When $d=2$ or $d=3$ the graph $G$ is isomorphic to that in
Example~\ref{game2} above, but for $d\geq 4$ the graphs are different.
Existence of multiple hard-core distributions on $\Z^{d-1}_{\text{bcc}}$ will
imply existence of draws on $\Z^d_{\text{bcc}}$.
\end{exa}

\begin{exa}\label{game4}
Let $\Out(x)=\{x+\sum_{i\in S} e_i: S\subset\{1,\dots,d\} \text{ with } 1\leq
|S| \leq d-1\}$. So a move of the game corresponds to incrementing at least
one, but not all, of the coordinates by one. Then $|\Out(x)|=2^d-2$.

Conditions (A1) and (A2) hold with $m=d$ if we set $S_k=\{x:\sum x_i=k\}$ and
$\phi(x)=x+e_1+e_2+\dots+e_d$.  For $d=2$ the game is the same as ever. For
$d>2$ there are some new features. For the first time we have $m>2$, and the
graph $G$ is not bipartite; from a given starting vertex, there are vertices
that can be reached when it is either player's turn. The graph $D$
is $(d-1)$-dimensional and $d$-partite. For $d=3$, it corresponds to the
triangular lattice.  For example, the map
\begin{equation}
\label{trihexmap}
(x_1, x_2, x_3) \to x_1(1,0)+x_2\bigl(-\tfrac12, \tfrac{\sqrt3}2\bigr)+
x_3\bigl(-\tfrac12, -\tfrac{\sqrt3}2\bigr)
\end{equation}
is an isomorphism from $D_k=S_k\cup\dots\cup S_{k+m-1}$ to the triangular
lattice for each $k$.
\end{exa}

\begin{figure}
\begin{center}
\begin{minipage}{0.35\textwidth}
\includegraphics[width=\textwidth]{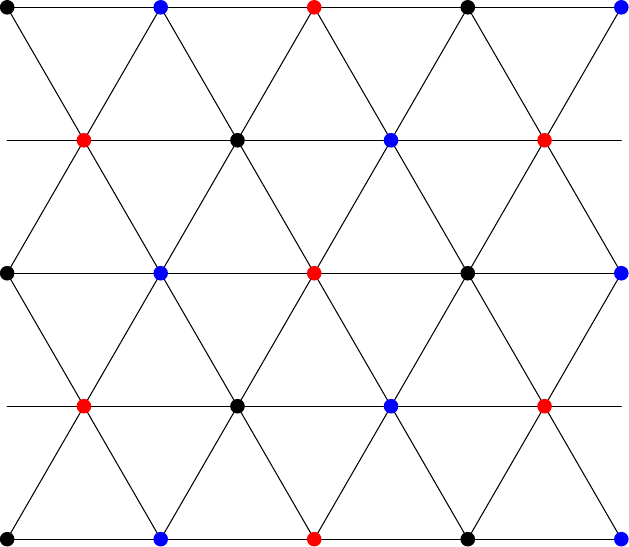}
\end{minipage}
\quad
\begin{minipage}{0.35\textwidth}
\includegraphics[width=\textwidth]{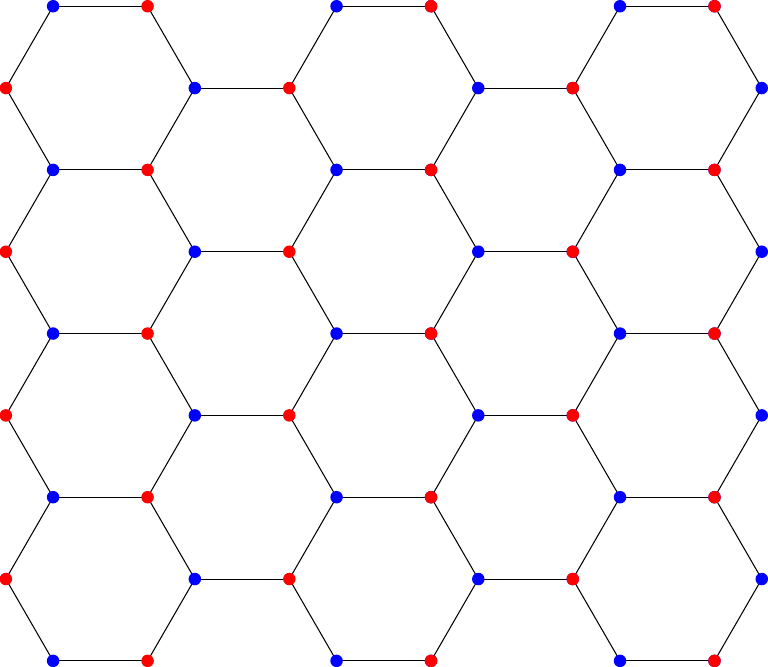}
\end{minipage}
\end{center}
\caption{Triangular lattice and hexagonal lattice.}
\label{fig:triangular}
\end{figure}

\begin{exa}\label{game5}
Fix $r$ with $1\leq r\leq d$,
and now restrict to sites $x\in\Z^d$ such that
$\sum x_i\equiv 0 \text{ or } r\bmod d$.
For $x$ with $\sum x_i\equiv 0 \bmod d$, let
$\Out(x)=\{x+\sum_{i\in S}e_i,
\text{ for any } S \subset\{1,\dots,d\} \text{ with } |S|=r\}$.
Meanwhile, for
$x$ with $\sum x_i\equiv r\bmod d$, let $\Out(x)=\{x+\sum_{i\in S}e_i, \text{
for any } S \subset \{1,\dots,d\} \text{ with } |S|=d-r\}$.

Now $|\Out(x)|=\ch{d}{r}$ for all $x$. Replacing $r$ by $d-r$ gives an
isomorphic graph, so we may assume $1\leq r\leq d/2$. Then conditions (A1)
and (A2) hold with $m=2$, with $\phi(x)=x+\sum_{i=1}^d e_i$, and with
$S_k=\{x:\sum x_i=dk/2\}$ for even $k$ and $S_k=\{x:\sum x_i=d(k-1)/2+r\}$
for odd $k$.

For $d=2$ (and hence $r=1$) the game is the familiar two-dimensional game.
For $d=3$ and $r=1$, we get $|\Out(x)|=3$ and the doubling graph $D$ is the
two-dimensional hexagonal lattice; this is the image of $\{x\in\Z^d: \sum
x_i\equiv 0 \text{ or } 1\bmod 3\}$, with edges between $x$ and $y$ where
$y\in \Out(x)$, under the map (\ref{trihexmap}) above.

For $d=4$ and $r=2$, the graph $G$ is isomorphic to the $d=4$ case of
Example~\ref{game2} above, and so $D$ is the standard cubic lattice $\Z^3$.
For $d=4$ and $r=1$, we have $|\Out(x)|=4$, and $D$ is the so-called
\df{diamond cubic graph} (see for example Section 6.4 of
\cite{ConwaySloane}). This graph may, for example, be represented as
\[
\bigl\{y\in\Z^3: y_1\equiv y_2\equiv y_3\bmod 2
\text{ and } y_1+y_2+y_3\equiv 0\text{ or }1\bmod 4
\bigr\},
\]
with edges between nearest neighbours (which are at distance
$\sqrt{3}/4$).
This is the image of $\{x\in\Z^d: \sum x_i\equiv 0 \text{ or }
1\bmod 3\}$, with edges between $x$ and $y$ where $y\in \Out(x)$,
under the map
\[
\begin{pmatrix}
y_1\\y_2\\y_3
\end{pmatrix}
=
\begin{pmatrix}
\phantom{-}x_1-x_2-x_3+x_4\\
-x_1+x_2-x_3+x_4\\
\phantom{-}x_1+x_2-x_3-x_4
\end{pmatrix}
\]
(see \cite{NagyStrand}). See Figure \ref{fig:diamond} for an illustration.
\end{exa}
\begin{figure}
\begin{center}
\begin{minipage}{0.45\textwidth}
\vspace{1cm}
\includegraphics[width=\textwidth]{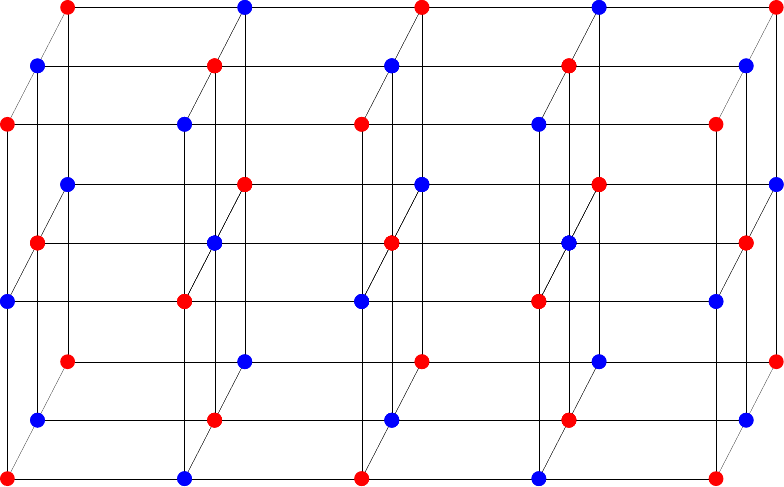}
\vspace{0.5cm}
\end{minipage}
\quad
\begin{minipage}{0.45\textwidth}
\includegraphics[width=\textwidth]{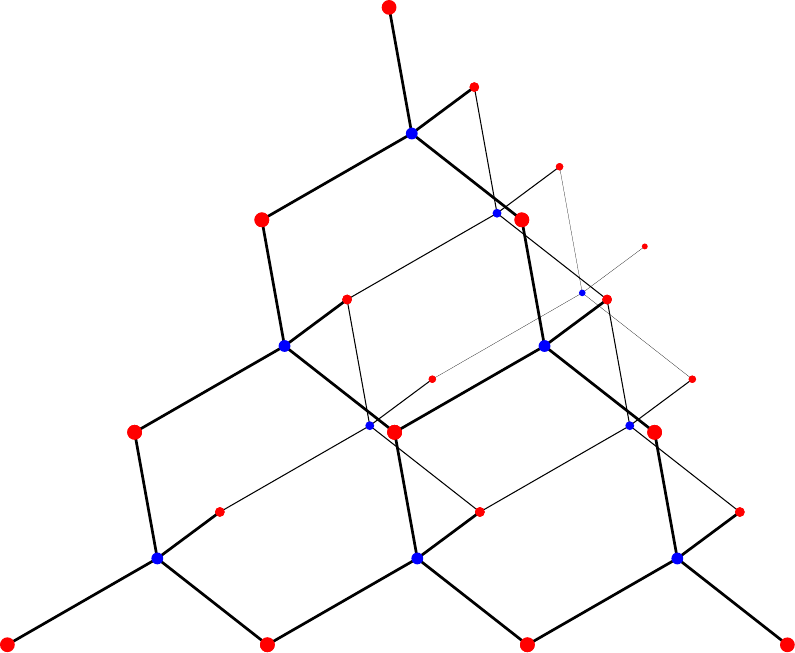}
\end{minipage}
\end{center}
\caption{The cubic lattice and the diamond cubic graph.}
\label{fig:diamond}
\end{figure}

\begin{thm}
\label{thm:higherd} There is positive probability of a draw from every vertex
for sufficiently small $p$ in the following cases: Example \ref{game2} for
all $d\geq 3$; Example \ref{game3} for $d=3$ and $d=4$; Example \ref{game4}
for $d=3$; and Example \ref{game5} for $d=3$ (with $r=1$) and $d=4$ (with $r=1$
 or $r=2$).
\end{thm}
\begin{proof}
In the cases listed, it is known that there exist multiple Gibbs
distributions for the hard-core model on the associated graph $D$ when the
activity parameter $\lambda$ is sufficiently high.  For the standard cubic
lattice in any dimension greater than 1, the result goes back to Dobrushin
\cite{dobrushin65}. Other models in two and three dimensions were covered by
Heilmann \cite{heilmann} and Runnels \cite{runnels}, including the triangular
and hexagonal lattices in two dimensions and the body-centered cubic lattice
and the diamond cubic graph in three dimensions.

Theorem \ref{thm:hardcoredraw} shows that there is positive probability of a
draw for small enough $p$ from some vertex, and since all the graphs $G$ are
vertex-transitive, the conclusion holds for every vertex.
\end{proof}

It is expected that in fact the hard-core model on $D$ has multiple Gibbs
distributions for $\lambda$ sufficiently large
in all of Examples \ref{game2}--\ref{game5}
whenever $d\geq 3$ (so that $D$ has dimension at least $2$).  This could likely be proved by Peierls contour arguments, although this requires a
suitable definition of a contour, which is typically graph-dependent, and
less straightforward than in other settings such as the Ising model.  Via
Theorem \ref{thm:hardcoredraw}, such non-uniqueness would imply
existence of draws for the corresponding graphs $G$.

We emphasize again that there is a more fundamental obstacle to proving
existence of draws for the standard oriented lattice $\Z^d$ of
Example~\ref{game1}, in that our condition (A2) does not hold here.

\subsection{Extending the hard-core correspondence}
\label{sec:hardcore-extension}
Various further extensions can be made while still preserving
the correspondence to the hard-core model. For example,
in the class of models considered in Theorem
\ref{thm:hardcoredraw}, we can augment the set of
allowable moves from site $x$ to include the point $\phi(x)$ itself.

\begin{samepage}
Specifically, replace (A1) and (A2) by the following assumptions:
\begin{itemize}
\item[(A$1'$)] For all $x\in S_k$, $\Out(x)\subset
    S_{k+1}\cup\dots\cup S_{k+m-1} \cup\{\phi(x)\}$.
\item[(A$2'$)] There is a graph automorphism $\phi$ of
    $G$ that maps $S_k$ to $S_{k+m}$ for every $k$, such
    that $\phi(x)\in\Out(x)$ for all $x$, and such that
    $\Out(x)\setminus\{\phi(x)\} =\In(\phi(x))\setminus
    x$.
\end{itemize}
\end{samepage}

Define $D$ as before; $D$ is a graph
isomorphic to any $D_k$, where $D_k$ is the graph with
vertex set $S_k\cup\dots\cup S_{k+m-1}$ and an undirected
edge $(x,y)$ whenever $(x,y)$ is a directed edge of $V$.
Note now an edge $(x,\phi(x))$ in $G$ does not give
a corresponding edge in $D$.

\begin{prop}
\label{prop:extendedhardcoredraw} Suppose that the graph $G$ satisfies
(A$1'$) and (A$2'$). If there are multiple Gibbs distributions for the
hard-core model on $D$ with activity $\lambda<1$. then the percolation game
on $G$ with $p=1-\lambda$ has positive probability of a draw from some
vertex.
\end{prop}

The method of proof is a slight variation on that of Theorem \ref{thm:hardcoredraw},
which we indicate briefly.
To reflect the presence of the edge $(x,\phi(x))$,
we change the hard-core update procedure.
When we perform an update of the vertex class $W_i$,
we now add that any vertex $v\in W_i$ which is in state 1 before the update must
move to state 0 after the update; otherwise the update at $v$ proceeds as before.

Again one can show that hard-core Gibbs distributions are stationarity under
such updates, but now with the activity parameter $\lambda$ equal to $1-p$
rather than $1/p-1$ as previously. (To verify the stationarity, one can start
by checking the detailed balance condition for an update at a single site;
then if the distribution is stationary for the update at any single site, it
is also invariant under simultaneous updates at any set of non-neighbouring
sites.)

Note that now when $p\to0$, we have $\lambda\to1$ rather than
$\lambda\to\infty$. Hence to show existence of draws for some $p$, we need
multiplicity of Gibbs distributions for some $\lambda<1$. For the case of the
standard cubic lattice, Galvin and Kahn \cite{galvin-kahn} show that this
holds for sufficiently high dimension, so that we can deduce the existence of
draws for the variant of Example \ref{game2} in which $\Out(x)=\{x\pm
e_i+e_d:1\leq i\leq d-1\}\cup\{x+e_d\}$, when $d$ is sufficiently large.

\section{Open questions}
\label{sec:open}

\subsection{The trapping game on the oriented cubic lattice}
For the trapping game on $\Z^d$ (where the allowed moves from site $x$ are
to any open site $x+e_i$ for $1\leq i\leq d$), do there exist any $p\in(0,1)$
and $d\geq 3$ for which draws occur with positive probability?

\subsection{Percolation games in higher dimensions}
Demonstrate positive probability of draws for percolation games
with some $q>0$ on a natural class of lattices in dimension $d\geq 3$.

\subsection{Monotonicity and phase transition}\label{open-mono}
Consider settings where draws are known to occur -- for example, the trapping game on the even sublattice of $\Z^d$ with $d\geq 3$ and moves allowed from $x$ to any open $x+e_d\pm e_i$ for $1\leq i\leq d-1$. Is the probability of a draw starting from the origin non-increasing in the density $p$ of traps? Or,
at least, is the set of $p$ that have positive draw probability a single interval containing $0$ (so that there is a single critical point at the
upper end of the interval)? If so, what happens at the critical point?
In the light of the connection to the hard-core model, such questions
are probably difficult -- it is unknown, for example, whether the
hard-core model on $\Z^d$ has a single critical
point for uniqueness of Gibbs distributions (see e.g.\
\cite{bgrt} for discussion and recent bounds).

\subsection{Converse statement of Theorem \ref{thm:hardcoredraw}\label{open-hard-core}}
As discussed in Section \ref{sec:converse}, the converse statement
of Theorem 2 holds when $m=2$; that is, when we have a bipartite
``doubling graph" $D$, the uniqueness of the hard-core Gibbs measure
on $D$ is equivalent to zero probability of draws in the trapping game.
Is this true also for $m>2$?
%
\subsection{Computation of winning probabilities}
Can one compute exact winning probabilties for games other than
in the case of the trapping game on $\Z^2$ (see \eqref{win-p})?
For example, as mentioned in the introduction, such questions for
the more general percolation game on $\Z^2$ are related to
the computation of generating functions for directed
animals, enumerated according to their area and perimeter.

\subsection{Mis\`ere games}
\label{subsection:misere}
In the target game (i.e.\ the percolation game with $p=0$ and $q>0$)
the first player to move to a marked site (i.e.\ a target) wins.
This can be seen as a mis\`ere version of the trapping game (i.e.\ the percolation
game with $q=0$ and $p>0$), where the first player to move to a
marked site (i.e.\ a trap) loses. There are also other natural mis\`ere variants of the trapping game.
For example, suppose that each site is marked with probability $p$ and
unmarked with
probability $1-p$, and that moves are allowed from $x\in\Z^2$ only to an
unmarked site in $\{x+e_1,x+e_2\}$, but now declare that if \emph{both} these sites are
marked, then the player whose turn it is to move \emph{wins}.  Does this game have zero probability of a draw for all $p$? Here we don't know of any useful
correspondence to a PCA with alphabet $\{0,1\}$.

\subsection{Elementary probabilistic cellular automata}
Is every elementary PCA (i.e.\ one with $2$ states and a size-$2$ neighhborhood)
on $\Z$ with positive rates ergodic?  Can our weighting approach be extended to
prove ergodicity for other PCA in this class?

\subsection{Undirected lattices}\sloppy
The following game is considered in \cite{BHMW}.
Each site of $\Z^d$ is independently a trap
with probability $p$, and two players alternately move a token.
From an open site $x$, a move is permitted to \emph{any} open nearest neighbour
 $x\pm e_i$
\emph{provided it has not been visited previously}. A
player who cannot move loses. This game is closely related
to maximum matchings, and this is used in \cite{BHMW} to
derive results for biased variants in which odd and even
sites have different percolation parameters.  However, for
the unbiased version described above it is unknown whether
there exist any $p>0$ and $d\geq 2$ for which draws occur
with positive probability. One can also extend the model to
include targets as well as traps.

\section*{Acknowledgments}
JBM was supported by EPSRC Fellowship EP/E060730/1.
IM was supported by the
Fondation Sciences Math\'{e}matiques de Paris.
We are grateful to the referee of an earlier version
for valuable comments. In particular, the discussion
in Section \ref{sec:converse} was prompted by the referee's
observation that such an approach gives a straightforward proof of
the $q=0$ case of Theorem \ref{thm:2d}.
\bibliographystyle{alpha}
\bibliography{biblio}
\end{document}